\documentclass[twoside,a4paper,12pt,centertags]{amsart}
\usepackage[T1]{fontenc}
\usepackage[utf8]{inputenc}
\usepackage{subfiles}
\usepackage{multicol}
\makeatletter
\usepackage[section]{placeins}
\AtBeginDocument{%
	\expandafter\renewcommand\expandafter\subsection\expandafter{%
		\expandafter\@fb@secFB\subsection
	}%
}
\makeatother
\makeatletter
\@namedef{subjclassname@2020}{%
	\textup{2020} Mathematics Subject Classification}
\makeatother
\usepackage{amsfonts,amssymb,amsmath,amsthm} 
\usepackage{mathrsfs}
\usepackage{nicefrac}
\usepackage{latexsym}
\usepackage{mathtools}
\usepackage[noadjust]{cite}
\usepackage{paralist}
\usepackage{longtable}
\usepackage{float}
\usepackage{tikz, tikz-cd}
\usetikzlibrary{patterns}
\usepackage{hhline}
\usepackage{epsf,epsfig,esint}
\usepackage{import}
\usepackage{enumitem}
\usepackage[colorlinks=true,citecolor=magenta,linkcolor=cyan,backref]{hyperref}
\AtBeginDocument{}

\numberwithin{equation}{section}

\theoremstyle{plain}
\newtheorem{thm}{Theorem}[section]
\newtheorem{lem}[thm]{Lemma}

\newtheorem{prop}[thm]{Proposition}

\theoremstyle{definition}
\newtheorem{defn}[thm]{Definition}
\newtheorem{rem}[thm]{Remark}

\usepackage{color}

\definecolor{gr}{rgb}   {0., 0.8, 0. } 
\definecolor{bl}{rgb}   {0., 0.5, 1. } 
\definecolor{mg}{rgb}   {0.7, 0., 0.7} 

\newcommand{\Bk}{\color{black}}

\renewcommand{\div}{\operatorname{div}}

\newcommand{\IC}{\mathbb{C}}

\newcommand{\IN}{\mathbb{N}}

\newcommand{\R}{\mathbb{R}}

\newcommand{\cL}{\mathcal{L}}

\newcommand{\cQ}{\mathcal{Q}}
\newcommand{\one}{{{\bf 1}}}
\newcommand{\n}{\Vert}

\newcommand{\loc}{\operatorname{loc}}

\renewcommand{\L}{\operatorname{L}} 
\newcommand{\C}{\operatorname{C}} 
\renewcommand{\H}{\operatorname{H}} 
\newcommand{\W}{\operatorname{W}}
\newcommand{\T}{\operatorname{T}}
\newcommand{\B}{\operatorname{B}}
\newcommand{\F}{\operatorname{F}}
\newcommand{\E}{\operatorname{E}}%

\newcommand{\V}{\operatorname{V}}
\newcommand{\PW}{\operatorname{PW}}

\DeclareRobustCommand{\Hdot}{\dot{\H}\protect{\vphantom{H}}} 

\newcommand{\I}{\mathrm{\, I \, }}

\renewcommand{\d}{\mathrm{d}} 
\newcommand{\eps}{\varepsilon} 
\renewcommand\Re{\operatorname{Re}}

\def\angle#1#2{\langle #1,#2 \rangle} 

\def\tpsi{\tilde\psi}

\newcommand\reu{\R_{+}\times \R^n}
\newcommand\ree{\R^{n+1}}

\def\Xint#1{\mathchoice
{\XXint\displaystyle\textstyle{#1}}%
{\XXint\textstyle\scriptstyle{#1}}%
{\XXint\scriptstyle\scriptscriptstyle{#1}}%
{\XXint\scriptscriptstyle%
\scriptscriptstyle{#1}}%
\!\int}
\def\XXint#1#2#3{{\setbox0=\hbox{$#1{#2#3}{%
\int}$ }
\vcenter{\hbox{$#2#3$ }}\kern-.6\wd0}}
\def\barint{\,\Xint -} 
\def\bariint{\barint_{} \kern-.4em \barint}
\def\bariiint{\bariint_{} \kern-.4em \barint}
\renewcommand{\iint}{\int_{}\kern-.34em \int} 
\renewcommand{\iiint}{\iint_{}\kern-.34em \int} 

\title[Parabolic SPDE with bounded measurable coefficients]{Stochastic and deterministic parabolic equations with bounded measurable coefficients in space and time: well-posedness and maximal regularity}
\author{Pascal Auscher}
\address{Universit\'e Paris-Saclay, CNRS, Laboratoire de Math\'{e}matiques d'Orsay, 91405 Orsay, France}
\email{pascal.auscher@universite-paris-saclay.fr}
\author{Pierre Portal}
\address{Pierre Portal
Mathematical Sciences Institute
\\
Australian National University
\\
Ngunnawal and Ngambri Country
\\
Canberra ACT 0200, Australia}

\email{pierre.portal@anu.edu.au}
\thanks{This work started during a visit of the first author to the ANU. The author thanks the ANU as well as the  ANU-CNRS  LIA FuMa for support. A CC-BY 4.0 \url{https://creativecommons.
org/licenses/by/4.0/} public copyright license has been applied by the authors to the present document and will be applied to all subsequent
versions up to the Author Accepted Manuscript arising from this
submission. The authors thank Emiel Lorist and Mark Veraar for help with bibliographic issues. 
}
 \date{October 11, 2023}
\keywords{parabolic stochastic PDE, Cauchy problems, Lions operator,  square functions, tent spaces}
\subjclass[2010]{Primary:  60H15 Secondary: 42B37, 35B65, 35B30.} 
\begin{document}

\maketitle

\begin{abstract} 
We establish well-posedness and maximal regularity estimates for linear parabolic SPDE in divergence form involving random coefficients that are merely bounded and measurable in the time, space, and probability variables. To reach this level of generality, and avoid any of the smoothness assumptions used in the literature, we introduce a notion of  pathwise  weak solution and develop a new harmonic analysis toolkit. The latter  includes techniques to prove the boundedness of various maximal regularity operators on relevant spaces of square functions, the parabolic tent spaces $\T^{p}$. Applied to deterministic parabolic PDE in divergence form with real coefficients, our results also give the first extension of Lions maximal regularity theorem on $\L^{2}( \R_{+} \times \R^{n})=\T^{2}$ to $\T^{p}$, for all $1-\varepsilon<p\le \infty$ in this generality. 
\end{abstract}

\tableofcontents

\section{Introduction} 
In this paper, we develop a well-posedness theory for stochastic parabolic Cauchy problems on $\R_{+}=(0,\infty)$ of the form
\begin{equation}\label{eq:SPDE}
\left\{
  \begin{array}{ll}
  \d U(t) - \div a(t,.) \nabla U(t) \, \d t & = \div F(t) \, \d t +  g(t)\, \d W_{\H}(t), \quad t\in \R_{+}\\
    U(0) & = \psi,
  \end{array}
\right.
\end{equation}
driven by a cylindrical Brownian motion $W_{\H}$ on a  complex separable  Hilbert space $\H$, with   $a \in \L^{\infty}(\Omega \times\R_{+}; \L^{\infty}(\R^{n}; \mathcal{L}(\R^{n})))$ (non necessarily symmetric)  and satisfying the uniform ellipticity condition:  for some $C>0$, 
\begin{equation*}
 a (\omega,t,x) \xi \cdot\xi  \geq C |\xi|^{2} \quad \forall \xi \in \R^{n},
\end{equation*}
for almost every $ (\omega,t,x)  \in \Omega  \times \R_{+} \times \R^{n}$, $\Omega$ being a probability space on which we realise  the cylindrical Brownian motion.  We treat initial data $\psi \in \L^{p}$ for $p\in(1,\infty)$, and forcing terms $F,g$ in appropriate $\L^p$ spaces of square functions, and    $g$ adapted to the underlying filtration (see below). \\

For Problem \eqref{eq:SPDE}, this constitutes the first result, even in  the deterministic  case (\textit{i.e.,} $g=0$ and $a,F, \psi$ constant in $\omega$), that can treat $\L^{\infty}$ coefficients in space and time when the data is not in a Hilbert space. All currently available theories, including state of the art results such as \cite{pronk-veraar, pv, dong-kim, dong21, dong22, kim-kim, nvwG} require some regularity in the space variable $x$ (at least some smallness of mean oscillations). These linear theories provide the foundation for a wide range of non-linear results, see \textit{e.g.,} \cite{dmh, ow, dong-krylov, dong-phan2, agresti-veraar}. Here we introduce an alternative foundation, free of the technical limitations created by the regularity assumptions in the linear theory. \\

Our approach brings to the problem the full force of parabolic PDE theory, as well as the modern developments of harmonic analysis beyond Calder\'on-Zygmund theory. This is possible thanks to the introduction of a flexible notion of pathwise weak solution, the use of function spaces adapted to problems with rough coefficients, and the proof of key boundedness results for maximal regularity operators.

\subsection{Square function spaces} 
Harmonic analysis has underpinned stochastic maximal regularity theory since the ground-breaking  papers of Krylov \cite{k1,k2, krylov}. In these papers,  Krylov established the following stochastic maximal regularity for the heat equation
\begin{equation*}
\mathbb{E} \bigg\| (t,x) \mapsto \bigg( \int _{0} ^{t} \nabla e^{(t-s) \Delta}g(s,.)\, \d W_{\H}(s) \bigg)(x)\bigg\|^{p} _{\L^{p}(\R_{+}\times \R^{n};\mathbb{C}^{n})}  
 \lesssim \mathbb{E}\|g\|^{p} _{\L^{p}(\R_{+}\times \R^{n};\H)},
\end{equation*} 
for all $p \ge 2$ and all adapted $g \in\L^{p}(\R_{+}\times \R^{n};\H)$.
 To prove such results, one uses extrapolation methods, such as Rubio de Francia's weighted extrapolation (see the survey \cite{krylov-survey}), to pass from results in $\L^{2}(\R^{n}; \L^{2}(\R_{+}))$ to mixed norm $\L^{p}(\L^q)$ spaces over space-time. This approach combines well 
with perturbative PDE methods, but has  two major limitations: it only works for $p \geq 2$ (the above  estimate fails for $p<2$), and it seems to require some regularity of the coefficients in the space variable $x$ to allow perturbative methods to work. \\

 Given these limitations, one can try other set-ups. 
For instance, one can use square function spaces of the form  $\V^{p}=\L^{p}(\R^{n}; \L^{2}(\R_{+}))$ ($\V$  stands for vertical square function) with norm $$\|g\|_{\V^p}=\bigg\| x\mapsto\bigg( \int_{0}^\infty |g(t, x)|^2\, \d t\bigg)^{1/2}\bigg\|_{\L^{p}}$$ while $\IC^n$ or $\H$ indicate  the  target space of the function, and at the same time we drop the source space $\R^n$ in the notation $\L^p(\R^n)$. 
The use of the $\V^p$ norm is natural for this problem because the It\^o isomorphism for stochastic integrals 
(see \cite{nvw}) gives that for all adapted process $f \in \V^{p}(\H)$, 
$$
\mathbb{E} \bigg\|  x\mapsto \bigg(\int _{0} ^{\infty} f(s,.)\, \d W_{\H}(s) \bigg)(x)\bigg\|_{\L^{p}} ^{p}
\eqsim \mathbb{E} \|f\|_{\V^{p} (\H )} ^{p}. 
$$
In these spaces, van Neerven, Veraar, and Weis have proven in \cite{nvwG} that
\begin{equation}
\label{eq:smr-heat}
\mathbb{E} \bigg\| (t,x) \mapsto \bigg( \int _{0} ^{t} \nabla e^{(t-s) \Delta}g(s,.)\, \d W_{\H}(s) \bigg)(x)\bigg\|^{p} _{\V^{p}}  
 \lesssim \mathbb{E}\|g\|^{p} _{\V^{p}(\H)},
\end{equation} 
for all  adapted processes  $g\in \V^{p}(\H)$ and all $p \in (1,\infty)$.
This removes the $p \geq 2$ limitation, but equations with $L^{\infty}$ coefficients in space and time still appear to be out of reach.

 Tent spaces  are an   alternative to the vertical square function spaces $\V^{p}$.    For $p \in (0,\infty)$, the tent space $\T^{p}$ is defined as the space of measurable functions $g$ with controlled  conical square function 
\begin{equation}
\label{def:Tp2}\|g\|_{\T^{p}} := \bigg( \int_{\R^{n}} 
\bigg( \int_{0}^{\infty} \fint_{B(x,\sqrt{t})} 
 |g(t,y)|^{2} \,{\d y \, \d t}\bigg)^{\frac{p}{2}}\, \d x \bigg)^{\frac{1}{p}}<\infty,
\end{equation}
where $B(x,r)$ denotes the Euclidean ball in $\R^n$ centred at $x$ with radius 
$r$,  and $\fint$ denotes the average for the Lebesgue measure. These spaces were introduced in \cite{cms} and have played a key role in harmonic analysis ever since, especially in the theory of boundary value problems on rough domains (see \textit{e.g.,} \cite{hkmp,aa}). The difference between the $\T^{p}$ norm and the $\V^{p}$ norm is the presence of an extra averaging over the ball $B(x,r)$, with $r=\sqrt{t}$ at ``time'' $t$ in accordance with parabolic scaling. \\ 

As both scales do not seem  fundamentally different, we may wonder what are the advantages and limitations of using  tent space norms on data and/or solutions.    

First of all, the spaces coincide for $p=2$, together with $\L^2(\R_{+};  \L^2(\R^n))$,  by Fubini's theorem, hence they share the same theory when $p=2$.  

Second, as function spaces,  $\V^{p}$ and $\T^{p}$ are continuously included in one another in the following way:
$\T^{p} \subsetneq \V^{p}$ for $p<2$ and
$ \V^{p} \subsetneq \T^{p}$ for $p>2$, see \cite{ahm}.

Third, free evolutions of  (regular) parabolic problems share common features in both scales.   
For instance,  classical  Littlewood-Paley estimates (see \cite{littleStein}) give  that, for all $p \in (1,\infty)$ and $f \in \L^{p}$, 
\begin{align}
\label{eq:LPest}
\|f\|_{\L^{p}} &\eqsim \|(t,x)\mapsto \nabla e^{t \Delta}f(x)\|_{\V^{p}} 
 \eqsim \|(t,x)\mapsto \nabla e^{t \Delta}f(x)\|_{\T^{p}}.
\end{align}

In fact, a first evidence that  $\T^p$ norms are better suited to the analysis  of equations whose coefficients lack regularity  is in  \cite{amp}. Indeed,  for deterministic real  parabolic equations $ \partial_{t}u-\div ( A  \nabla u)=0$ {in} $ \reu$,  $
 u(0,\cdot)= \psi \in \L^p$,     with $\L^{\infty}$ coefficients in space and time, a  weak solution is given by the propagator ({\it i.e.},\,  $u=\Gamma(t,0)\psi$) and it is proved that  
 \begin{align}
\label{eq:LPestA}
\|\psi\|_{\L^{p}} \eqsim \|(t,x)\mapsto (\nabla \Gamma(t,0) \psi)(x)\|_{\T^{p}} ,
\end{align}  
while no such bounds exist for the vertical square function norm when $p\ne 2$
 (and a uniqueness statement also holds, see Remark \ref{rem:CP}). 
This fact already implies that usage of tent spaces opens the door to having $\L^p$ initial data in Problem \eqref{eq:SPDE} (although we shall not use this estimate in the stochastic case). \\

Besides an appropriate notion of solution, what is missing as of now to understand Problem \eqref{eq:SPDE}   are  tent space estimates for the maximal regularity operators needed to analyze the source term and the noise term in  the  absence of regularity on the coefficients. We next describe  the operators that we consider.

\subsection{Maximal regularity operators}
To produce a solution  to Problem \eqref{eq:SPDE} in tent spaces,  we follow the approach used in \cite{pv}, 
decomposing the stochastic parabolic problem with random coefficients dependent on space and time into two parts: one stochastic heat equation, and a family of deterministic problems with time dependent coefficients in which  source terms in divergence form appear naturally and must be included in the study. This circumvents technical challenges coming from the fact that random propagators are typically not adapted (see \cite{pronk-veraar})  to the underlying filtration
and amounts to proving maximal regularity estimates for  two deterministic operators.\\

For the stochastic heat equation, this reduces to extrapolating the boundedness of a quadratic maximal regularity operator initially defined by
\begin{equation}
\label{eq:quadmaxheat}
\begin{array}{cccc}\cQ: & \L^2(\R_{+};  \L^2(\R^n)) & \rightarrow & \L^2(\R_{+};  \L^2(\R^n) ) \\ & f & \mapsto & \left[(t,x) \mapsto \big(\int _{0} ^{t} |\nabla e^{(t-s) \Delta}f(s,.)(x)|_{\IC^n}^{2} \, \d s\big)^{\frac{1}{2}}\right]\end{array}
\end{equation}

For  the  deterministic perturbations with $\L^{\infty}$ uniformly elliptic coefficients   $A$ in both the space and time variables, this  reduces to extrapolating boundedness of the  following  operator, initially defined and proved to be bounded by Lions in \cite{Lions}: 
\begin{equation}
\label{eq:Lionsop}
\left.\begin{array}{cccc}\cL: & \L^2(\R_{+};  \L^2(\R^n; \IC^n)) & \rightarrow & \L^2(\R_{+};  \L^2(\R^n; \IC^n)) \\ & F & \mapsto & \nabla u\end{array}\right.
\end{equation}
where $u$ is the  (energy)  solution (see below) of
\begin{equation*}
\begin{cases}
 \partial_{t}u-\div ( A  \nabla u)=\div F & \textrm{in}\ \reu 
 \\
 u(0,\cdot)= 0    & \textrm{in}\ \R^n,
\end{cases}
\end{equation*}

It should be noted that stochastic maximal regularity results give information about the gradient of the solution. This is different from standard deterministic maximal regularity, which gives information about second order spatial derivatives of the solution. Such information cannot be obtained for stochastic problems, because the irregularity of the Brownian motion prevents a gain of two order of smoothness. With stochastic applications in mind, it is thus natural to estimate the solutions of related deterministic problems using the Lions operator rather than   maximal regularity operators such as $f \mapsto [(t,x) \mapsto \int _{0} ^{t} \Delta e^{(t-s) \Delta}f(s,.)(x)\, \d s]$. This is good news as  such  maximal $\L^2$  regularity is unknown for our time dependent parabolic problems with $\L^\infty$ coefficients (and fails  in general abstract contexts), see the survey article \cite{survey}. \\

We  are able to extrapolate $\T^{p}$ boundedness from $p=2$ to $1-\eps<p\le\infty$ for the operator $\cL$, see Section \ref{sec:main}. In the context of $\L^{\infty}$ coefficients in space and time, we  thus  obtain the first boundedness result for the Lions operator outside of Hilbert spaces.  To do so, we introduce new techniques to overcome various harmonic analysis challenges, especially when $p<2$ as the other case $p>2$ is comparatively easy.  \\

We need to handle the fact that appropriate heat kernel bounds are available for the propagator $\Gamma(t,s)$, thanks to the fundamental papers of Nash \cite{nash} and Aronson \cite{aronson}, but not for its gradient. This prevents us from applying the tent space extrapolation theory for singular integral operators from \cite{akmp}. We circumvent this problem by using PDE and harmonic analysis techniques  such as (boundary) Caccioppoli's inequality and Whitney decompositions to exploit local   estimates of the form 
$$
\int_{0} ^{t} \int _{2^{j}B \backslash 2^{j-1}B} |\nabla  \Gamma(t,s)^{*} \varphi(y)|^{2}\,  \d y\, \d s 
\lesssim e^{-c4^{j}r^{2}/t}\int_{B} |\varphi(y)|^{2}\, \d y$$
for $\varphi$ supported in a ball $B$ of radius $r$,  $t<r^{2}$ and $j\ge 2$, see \eqref{eq:phiB2} below.  Such estimates can be interpreted as heat kernel bounds integrated in space and time. By contrast, the off-diagonal bounds normally used in tent space extrapolation theory (see \cite{akmp}) are heat kernel bounds integrated in space, but not in time.\\

However,  we have to face the fact that the decay in these heat kernel bounds is not sufficient to sum up the various local  pieces in our tent space analysis.  Instead of using weighted tent spaces as in \cite{anp}, we exploit here  the H\"older regularity estimates for weak solutions of parabolic real equations as established by Nash in \cite{nash}. Indeed,  $\L^{1}-\C^{\alpha}$ mapping properties of the propagators  provide more decay than merely $\L^{1}-\L^{\infty}$ bounds.  \\

As for  the quadratic maximal regularity operator $\cQ$, we extrapolate the  $\T^{p}$ boundedness from $p=2$ to  $1<p\le \infty$ in Section \ref{sec:quadratic}.     Again, the case $p>2$ is the easiest. Extrapolation  for $p<2$, however, requires another argument than for the Lions operator as we do not know how to reach $p=1$, and hence cannot use atomic decompositions. We use a Calder\'on-Zygmund decomposition in tent spaces  introduced by Huang in his PhD thesis, see \cite{yi}, and prove weak type $\T^p$ estimates. Such a  decomposition  is natural, but difficult to use, because of the lack of cancellation of local terms in this decomposition.  We manage to do so here by importing to the tent space setting ideas used in the $\L^p$ context by Blunck-Kunstmann, see \cite{bk} and also    \cite{a-mem}. We gain just enough decay by doing a further decomposition to force cancellation of local terms in such a way that we can exploit  results by Ne\v{c}as \cite{necas} or Bogovski\u{i} \cite{bog1,bog2} on inverting the divergence operator.\\

\subsection{Which solutions? }
Once  the harmonic analysis results matching the desired strategy have been established,  we need to look for a convenient notion of solutions for Problem  \eqref{eq:SPDE} that is compatible with tent spaces.\\  

There has been previous attempts at developing a tent space theory of stochastic maximal regularity. In \cite{anp}, with van Neerven, we considered time independent problems with $\L^{\infty}$ coefficients in the space variable. Besides the limitation to constant in time coefficients, and to the notion of solution used there,  another weakness of \cite{anp} is also the fact that it requires the tent space norm to be weighted by a factor $t^{-\beta}$ for some $\beta>0$. This, in turn, restricts the initial data that can be considered (to a Sobolev space rather than all of $\L^{p}$). 
The case of time dependent coefficients was then considered by the second author and Veraar in \cite{pv}. This included some estimates between weighted tent spaces for problems with $\L^{\infty}$ coefficients in space and time, but no maximal regularity estimates in that context.
\\

Starting from a different view-point, the weak formulation of a parabolic equation in the deterministic case has been very fruitful. It involves testing against functions over space-time. 
In the stochastic case, however, one cannot conveniently  interpret $g(s,.)\, \d W_{\H}(s)$ in terms of a duality pairing with a test function in space-time. This prevents us from using a weak formulation for  Problem \eqref{eq:SPDE}.  A traditional response is to test against functions of the space variable to express the equation as an ordinary  stochastic   differential equation. This requires a priori knowledge that the solution, at time $t$, belongs to some Banach space as a function of $x$. 
Consequently, one cannot use  the full force of the deterministic theory of parabolic problems. One can also consider fixed point formulations, as in \cite{anp}, but the corresponding solutions may not be stochastic processes, and their uniqueness theory is,  to some extent,  trivial.  To reach the level of flexibility of deterministic theories, 
 the next best thing is to use a formulation after integration by parts in time.  
Doing this  
at the level of the mild formulation  leads to the notion of pathwise mild solution  introduced  in \cite{pronk-veraar} for linear problems (see, in particular, \cite[Equation (4.6)]{pronk-veraar}), and in \cite{dhn} for quasilinear problems.   \\

In   the notion of {\bf pathwise weak solution} that we introduce,  we first test  against functions which depend on time and space variables  in the duality in space at a fixed time, and then integrate  by parts,  which is less demanding. In particular,  this can be done  even when the propagator associated with the equation $\partial_{t} u = \div a(t,.) \nabla u$ only defines an evolution family on $\L^2$.  Moreover, this allows us to genuinely consider the uniqueness problem.  \\

More precisely, for a large class of data $(\psi,F,g)$, a pathwise weak solution of Problem \eqref{eq:SPDE} is a random function
$U\in\L^1(\Omega; \L^1_{\loc}([0,+\infty); \W^{1,1}_{\loc}(\R^n)))$ such that, 
for all (non random) test functions  $\varphi\in \C^{\infty}_{c}([0,\infty)\times \R^{n})$, 
the random functions
$
 V_{\varphi} \colon t\mapsto \langle U(t,.),\varphi(t,.)\rangle - \big\langle \int  _{0} ^{t} g(s,.)\, \d W_{\H}(s), \varphi(t,.) \big\rangle
$
and
\begin{align*}
& \mathbb{V}_{\varphi} \colon t\mapsto \langle \psi, \varphi(0,.) \rangle 
-
\int  _{0} ^{t}  \langle a(s,.)\nabla U(s,.),\nabla \varphi(s,.) \rangle \, \d s +  \int  _{0} ^{t} \langle
U(s,.),\partial_{s}\varphi(s,.) \rangle\,  \d s \\ &
\qquad  \quad -  
\int  _{0} ^{t} \langle
F(s,.),\nabla \varphi(s,.) \rangle\,  \d s  
- \int  _{0} ^{t} \bigg\langle \int  _{0} ^{s} g(\tau,.)\, \d W_{\H}(\tau), \partial_{s}\varphi(s,.) \bigg\rangle\,  \d s,
\end{align*} 
 agree in $\L^{1}(\Omega\times [0,T])$ for all $T>0$. 
 See Definition \ref{def:pathweak} and Lemma \ref{lem:pathweak}  for the definition of all the terms in the above expressions. \\

The notion of pathwise weak solution is  inspired by the notion of pathwise mild solution, and is  also related to the notion of kinetic solution
first introduced for hyperbolic conservation laws in \cite{lpt}, and then used for parabolic SPDE in \cite{hof,dhv}. The two notions share the idea of using test functions over space-time, and to first interpret solutions as random functions rather than stochastic processes (see \cite[Remark 2.3]{hof} and  Lemma \ref{lem:pathweak}). One can expect  a notion of solution combining the advantages of pathwise weak and kinetic solutions to be used 
 in future treatments of quasilinear problems with rough coefficients.
  
\subsection{Well-posedness} Having the notion of pathwise weak solution allows us to prove our well-posedness result  Theorem \ref{thm:spde}  in Section \ref{sec:stoch}.  

For $1<p<\infty$, and data 
$(\psi, F,g)$ with $g$ adapted, such that $$\mathbb{E}\|\psi\|_{\L^{p}} ^{p} +
\mathbb{E}\|F\|_{\T^{p}(\IC^n)} ^{p} + \mathbb{E}\|g\|_{\T^{p}(\H)} ^{p}<\infty$$ we prove existence and uniqueness of pathwise weak solutions in the space $\L^{p}(\Omega; \dot{\mathcal{T}}^{p}_{1})$, where $$
\dot{\mathcal{T}}^{p}_{1}:= \{u \in  \L^1_{\loc}([0,\infty)\times \R^{n})  \;;\; \nabla u \in \T^{p}(\IC^{n})\}.$$

Using the results from Section \ref{sec:main} and Section \ref{sec:quadratic} together with the  Littlewood-Paley estimate \eqref{eq:LPest},  we show that our pathwise weak solutions satisfy the stochastic maximal regularity estimate: 
\begin{equation}
\label{eq:maxrefTp}
\mathbb{E}\|\nabla U\|_{\T^{p}(\mathbb{C}^{n})} ^{p} \lesssim \mathbb{E}\|\psi\|_{\L^{p}} ^{p} +
\mathbb{E}\|F\|_{\T^{p}(\IC^n)} ^{p} + \mathbb{E}\|g\|_{\T^{p}(\H)} ^{p}.
\end{equation}
We also obtain  $\T^{p}$  estimates for $U$ itself on any finite time interval $[0,T]$. 

As $\T^{2}=\L^{2}(\R_{+}; \L^2(\R^n))$ isometrically, \eqref{eq:maxrefTp}  corresponds, for $p = 2$, to  classical stochastic maximal regularity  estimate in  the energy space $\L^2(\Omega; \L^{2}(\R_{+}; \Hdot^1(\R^n)))$.   

Uniqueness of pathwise weak solutions in $\L^{p}(\Omega; \dot{\mathcal{T}}^{p}_{1})$  is a key feature of our result, as the uniqueness problem is far from trivial in this context (as opposed to the context of mild solutions).   We exploit for this the uniqueness result 
 for  deterministic equations in our earlier work with Monniaux \cite{amp} when $p=2$, followed by its extension to general $p$ by Zato\'n  \cite{zaton}.  \\

Let us finish by commenting on the regularity of our solutions.  
In contrast to mild or even kinetic solutions, our solutions are typically not (almost surely) contained in a space of continuous functions valued in a Banach space, and even  $t \mapsto \langle U(t,.),\varphi(t,.) \rangle$  may not be almost surely continuous when $p \neq 2$ (if $p=2$,  it is the case and follows from the construction of the It\^o integral in $\L^2$). This is another reason while Problem \eqref{eq:SPDE} does not have a standard weak formulation. 
The singularity is, in fact, carried by the noise term
 appearing in  $V_{\varphi}$  as it  is merely defined by density of simple adapted processes in $  \L^{p}(\Omega;\T^{p}(\H))$,  see Lemma \ref{lem:Smap}. All terms  appearing in $\mathbb{V}_{\varphi}$,  on the other hand, are almost surely absolutely continuous, even the one involving an integrated version of the noise. 
  This is reminiscent of rough path theory, where one considers a noise term in $\C^{\alpha}$, and an enhancement of the noise in $\C^{2\alpha}$ (typically through an iterated integral; see \cite{fh}). For this reason, we suggest to call  $$(t,x) \mapsto U(t,x)- \bigg(\int  _{0} ^{t} g(s,.)\, \d W_{\H}(s)\bigg)(x)$$
the {\em enhanced part of the solution} as it carries more regularity. 
 This  enhanced part tested against any $\varphi$ has an almost surely absolutely continuous trajectory, while  it  may not be the case for the solution itself. \\

\subsection{Summary} 
The global road map to have in mind  has three components:  
\begin{itemize}
\item estimates for deterministic problems in appropriate square function spaces such as 
tent spaces  (see Theorems \ref{thm:main1}, \ref{thm:main2}, \ref{thm:quad1}, \ref{thm:quad2}), 
  \item  the  introduction and use of the notion of pathwise weak solution to obtain well-posedness  (see Definition \ref{def:pathweak}), 
\item a stochastic integration theory that  implements such estimates for adapted processes $g$  (see Lemma \ref{lem:Smap}), 

\end{itemize}

We thus hope that it can be used as a foundation for a stochastic maximal regularity theory adapted to much more general (and potentially singular) noise terms.  This  would only require appropriate generalisations of the last point,  
in the spirit of rough path theory and its modern developments (as in, for instance, \cite{h, gh, bm}).  \\

{\em Notation:} To simplify notation, we typically only indicate the target space of our function spaces, writing $\L^{p}$ instead of $\L^{p}(\R^{n})$, and $\T^{p}(\H),\V^{p}(\H)$ for square function spaces over $[0,\infty) \times \R^{n}$ with  $\H$-valued functions. We indicate the space of variables (such as in $\L^2(\Omega; \L^{2}(\R_{+}; \Hdot^1(\R^n)))$ for instance) when it is needed for clarity. We use the Vinogradov notation $\lesssim , \eqsim$ to hide irrelevant constants.

\section{Review on Cauchy problems and the Lions operator.}
\label{sec:ing}

Let us first review some basic facts and introduce our notation. Throughout, we assume that $A=(a_{ij})$ is an 
$n\times n$-matrix of  bounded  measurable real or complex-valued functions  on $\ree$. We let 
 $\Lambda=\|A\|_{\infty}$. For the ellipticity, we assume a G\aa rding inequality: for some $\lambda>0$   
\begin{equation}
\label{eq:ell}
\Re \int_{\R^n}{A(t,x)\nabla w(x)}\cdot {\nabla w(x) }\, \d x \ge \lambda  \|\nabla w\|_{\L^{2}}^2.
\end{equation}
for every $w\in \Hdot^1(\R^n)$ (the closure of $\C_{c}^\infty(\R^n)$ for the semi-norm $\|\nabla w\|_{\L^{2}})$ and  almost every $t>0$. Here, ${z}\cdot{z'}$ denotes the complex inner product in $\IC^n$. 
The meaning of  $ \partial_{t}u-\div (A \nabla u)=\div F$ on an open bounded cylinder $(a,b)\times  \mathcal{O} $  is in the sense of weak solutions: $u$ belongs to $\L^2(a,b; \H^1( \mathcal{O} ))$ and the equation is verified tested against (complex-valued) functions in $\varphi\in \C_{c}^\infty((a,b)\times \mathcal{O} )$ in the sense
that
\begin{equation}
\label{eq:weakform}
-\int u \ \partial_{t}\overline\varphi  + \int   {A\nabla u}\cdot{ \nabla \varphi} = -\int  {F}\cdot{  \nabla \varphi}.
\end{equation}
Here $\H^1( \mathcal{O} )$ is the usual Sobolev space. Whenever $F\in \L^2(a,b; \L^2( \mathcal{O} ;\IC^n))$, it is well-known that $u\in \C([a,b], \L^2( \mathcal{O}'))$ for any open set $ \mathcal{O}'$ with $\overline{ \mathcal{O}'} \subset  \mathcal{O}$  \cite[Chap.~I, Prop.~3.1, and Chap.~II, Thm.~3.1]{Lions}.  They also satisfy local energy estimates called Caccioppoli inequality. We only need the ones when $F=0$   (see \textit{e.g.,} \cite[Prop.~3.6]{amp}),  which can be expressed as follows: for fixed $\alpha,\beta>1$ and whenever $b-\alpha r^2\ge a$ and $\overline{B(x,\beta r)}\subset  \mathcal{O} $, for some constant depending only on $\lambda,\Lambda, \alpha,\beta$,  
\begin{equation}
\label{eq:caccio}
\int_{b-r^2}^b\int_{B(x,r)} |\nabla u|^2 \lesssim r^{-2} \int_{b-\alpha r^2}^b\int_{B(x,\beta r)} | u|^2.
\end{equation}

We say that a   weak solution is global (in $\R_{+}\times \R^n$)  if it belongs to $\L^2_{\loc}(\R_{+}; \H^1_{\loc}(\R^n))$ and \eqref{eq:ell}  holds for all $\varphi\in \C_{c}^\infty(\R_{+}\times \R^n)$. We state the following lemma for later use in Section \ref{sec:stoch}.

\begin{lem} 
\label{lem:countable} There exists a countable subset $D$ of $\C_{c}^\infty(\R_{+}\times \R^n)$ such that, for any $u\in \L^2_{\loc}(\R_{+}; \H^1_{\loc}(\R^n))$ and $F \in \L^2_{\loc}(\R_{+}; \L^2_{\loc}(\R^n;\IC^n))$, if \eqref{eq:weakform} holds for all $\varphi\in D$,  then it  holds for all $\varphi\in \C_{c}^\infty(\R_{+}\times \R^n)$.
\end{lem}

\begin{proof} We can  write \eqref{eq:weakform} as $\int G\cdot \nabla_{t,x}\varphi=0$ with $$G=(-u, A\nabla u +F) \in \L^2_{\loc}(\R_{+}; \L^2_{\loc}(\R^n;\IC^{1+n})).$$  It is then an exercise to construct $D$. Take a smooth partition of unity $1=\sum_{k\ge 0} h^2_{k}$ where the $h_{k}$ have support in unit cubes of $\R^{1+n}$ with bounded overlap. Then, given an arbitrary $\varphi\in \C_{c}^\infty(\R_{+}\times \R^n)$, write 
$\varphi=\sum (\varphi h_{k})h_{k}$ and decompose $\varphi h_{k}$ in a Fourier series. Finally multiply the result by $\chi_{\ell}(t)$ for an appropriate integer $\ell$, where $\chi_{\ell}$ is a smooth function that is supported in $[0,\infty)$ and which is 1 on $[2^{-\ell},\infty)$, so that $\chi_{\ell}(t)=1$ on the support of $\varphi$. Plug this decomposition of $\varphi$ into \eqref{eq:weakform}. 
 Since $\varphi$ has compact support, the sum $\varphi=\sum (\varphi h_{k})h_{k}$ only involves finitely many terms.  Besides, smoothness of $\varphi h_{k}$ implies fast convergence of its  Fourier coefficients, so that one can pass the Fourier summation outside the integral. The collection $D$  can thus be defined as the set of all functions of the form $(t,x)\mapsto h_{k}(t,x)\chi_{\ell}(t)e^{i2\pi m\cdot(t,x)} $, for $k\ge0$, $\ell\ge 0$, $m\in \mathbb{Z}^{1+n}$.\end{proof}

The construction of the Lions operator, together with the estimates that come with it, follows from  the uniqueness result proved for local in time energy solutions. 
 Moreover, uniqueness holds in the largest possible homogeneous energy space  $\L^2(\R_{+};\Hdot^1(\R^n))$ and is a consequence of working on an infinite  time interval  (otherwise, this is false). We give further details below. This is an important point for the results we prove in this work.

\begin{thm}
\label{thm:CP} Given a source term $F\in \L^2(\R_{+};  \L^2(\R^n; \IC^n))$, and initial data $\psi\in \L^{2}$,  
the Cauchy problem \begin{equation}
\label{eq:CP}
\begin{cases}
 \partial_{t}u-\div (A \nabla u)=\div F & \textrm{in}\ \reu 
 \\
 u(0,\cdot)= \psi   & \textrm{in}\ \R^n
\end{cases}
\end{equation}
 is well-posed for global weak solutions in the class $\L^2(\R_{+};\Hdot^1(\R^n))$. 
 
 The solution  additionally belongs to 
$\C_{0}([0,\infty); \L^{2}(\R^{n}))$\footnote{The subscript 0 means null limit at $\infty$.},  and satisfies the following estimates. 
 When $F=0$, we have that: 
\begin{equation}
\label{eq:enestpsi}
\sup_{t>0}\|u(t)\|_{\L^{2}}^2 + 2\lambda \|\nabla u\|_{\L^2(\R_{+};\L^2(\R^n;\IC^n))}^2 \le   \|\psi\|_{\L^{2}}^2.
\end{equation}
 When $\psi=0$, we have that: 
\begin{align}
\label{eq:enestF}
\sup_{t>0}\|u(t)\|_{\L^{2}}^2 &\le \frac{1}{2\lambda}\|F\|_{\L^2(\R_{+};  \L^2(\R^n; \IC^n))} ^{2}, \\ 
  \|\nabla u\|_{\L^2(\R_{+};\L^2(\R^n;\IC^n))} &\le   \frac{1}{\lambda}\|F\|_{\L^2(\R_{+};  \L^2(\R^n; \IC^n))},
\end{align}
 The Lions operator given by \eqref{eq:Lionsop} is  thus  bounded on $\L^2(\R_{+};  \L^2(\R^n; \IC^n))$ with norm bounded by $\frac{1}{\lambda}$. 
\end{thm}

\begin{proof} There are several ways to see this. A first proof can be to observe that existence and estimates  are well-known on any finite time interval $[0,T]$ instead of $\R_{+}$ and go back to Lions, see \cite[Chap.~II, Thm.~3.1]{Lions}.  There,  thanks to the emdedding theorem of Lions that $\L^2(0,T; \H^1(\R^n)) \cap \H^1(0,T; \H^{-1}(\R^n)) \hookrightarrow \C([0,T];\L^2(\R^n))$, see \cite[Chap.~I, Prop.~3.1]{Lions},  uniqueness through integral identities  is established in the class $\L^2(0,T; \H^1(\R^n))$. As one can do this for all $T$, we obtain a global weak  solution with the above bounds.  In particular, the global weak solution belongs to $\L^2(\R_{+};\Hdot^1(\R^n))$. As it does not belong  to $\L^2(\R_{+};\L^2(\R^n))$, the Lions embedding  result  does not apply anymore. However, working on an infinite interval allows  us  to consider  the homogeneous version of the embedding proved in   \cite[Lem.~3.1]{amp}, which asserts that   $\L^2(\R_{+}; \Hdot^1(\R^n)) \cap \Hdot^1(\R_{+}; \Hdot^{-1}(\R^n)) \hookrightarrow \C_{0}([0,\infty); \L^{2}(\R^{n}))+\IC$.  Hence, the sought  global weak solutions in $\L^2(\R_{+};\Hdot^1(\R^n))$ when $F=0$ and $\psi=0$  are automatically within  $\C_{0}([0,\infty); \L^{2}(\R^{n}))$, and uniqueness follows again using  integral identities, see \cite[Thm. 3.11]{amp}. 

A second and direct proof of existence in  $\L^2(\R_{+};\Hdot^1(\R^n))\cap \C_{0}([0,\infty); \L^{2}(\R^{n}))$, starting from semigroup theory  in the case of time-independent coefficients and using weak convergence in the equation from integral identities to pass to the general case,  is   given  in \cite[Thm. 3.11]{amp} when $F=0$, and it is not  hard to pursue the same  method of proof when $F\ne 0$. 

A third and full proof of existence and uniqueness with a different method based on a variational approach is in \cite{ae}, Thm.~2.54 and Section 2.14. 
\end{proof}

Existence and uniqueness imply the construction of propagators which we review now (see \cite{ae}). There exists a unique family of contractions $\Gamma(t,s), -\infty<s\le t<\infty$ on $\L^{2}$, $ \Gamma(t,t)$ being the identity, defined by setting $\Gamma(t,s)\psi$ to be  the value in $\L^{2}$ of the solution at time $t$, with initial data $\psi$ at time $s$ (shifting the origin of time). 
 In this generality, their adjoints $\tilde\Gamma(s,t):=\Gamma(t,s)^*$ can be shown to be the propagators for the adjoint backward Cauchy problem $-\partial_{s}v-\div (A^*\nabla v)=0$ on $(-\infty,t)\times \R^n$ with given final data at time $t$. 
 
Let $F\in \L^2(\R_{+};  \L^2(\R^n; \IC^n))$. The solution  $u$ such that $\nabla u=\cL F$ given in Theorem~\ref{thm:CP} can be represented in $\L^{2}$ as
\begin{equation}
\label{eq:representation}
u(t, .)=  \int_{0}^t\Gamma(t,s)\div F(s, .)\, \d s \end{equation}
where the integral is understood  in the weak sense, that is, for all $\tpsi \in \L^{2}$,
\begin{equation}
\label{eq:weaksense}
\angle{u(t, .)}{\tpsi}= -  \int_{0}^t\angle{F(s, .)}{\nabla\tilde\Gamma(s,t)\tpsi}\, \d s. 
\end{equation}
with $\angle{.}{.}$ being here the canonical $\L^2(\R^n;\IC^d)$ inner product  (for $d=1,n$). 
 The last integral converges thanks to the inequality
 \begin{align*}
 \int_{0}^t|\angle{F(s, .)}{\nabla\tilde\Gamma(s,t)\tpsi}|\, \d s & \le \|F\|_{\L^2(\R_{+};  \L^2(\R^n; \IC^n))}\|\nabla\tilde\Gamma(\cdot,t)\tpsi\|_{\L^2(-\infty,t; \L^2(\R^n;\IC^n))}
 \\
 &
 \le \frac 1{\sqrt {2\lambda}} \|F\|_{\L^2(\R_{+};  \L^2(\R^n; \IC^n))}\|\tpsi\|_{\L^{2}}
\end{align*}
on applying to the backward solution $\tilde\Gamma(\cdot,t)\tpsi$ the estimates of Theorem \ref{thm:CP} reversing the sense of time. 

We shall also need  {\em $\L^2$-$\L^2$ off-diagonal bounds} for $\Gamma(t,s)$ holding at this level of generality (\cite{HK} or \cite{amp} or \cite{ae}):   
 there exists a constant $c>0$, depending only on $\lambda,\Lambda$ and dimension  such that for all Borel sets 
$E$, $F$ in $\R^n$, all $s<t$, and all $f\in \L^{2}$, we have that 
\begin{equation}
\label{eq:od}
\n \one_E \Gamma(t,s) \one_F f\n_{\L^{2}} \le 
e^{-c\frac{d(E,F)^2}{t-s}}  \n \one_F f\n_{\L^{2}},
\end{equation}
 with $d(E,F) := \inf\{|x-y|: \ x\in E, \ y\in F\}$.

The theory in \cite{ae} shows that there is an extension principle for parabolic problems in this context (thanks to causality). For $G\in \L^2(\R; \L^2(\R^n;\IC^n))$, the equation $\partial_{t}v-\div (A \nabla v)=-\div G$ on $\R \times \R^n$ has a weak solution, unique up to a constant, in the class $\L^2(\R; \Hdot^1(\R^n))$ and one has 
$\|\nabla v\|_{\L^2(\R; \L^2(\R^n;\IC^n))}\le \frac 1 \lambda \|G\|_{\L^2(\R; \L^2(\R^n;\IC^n))}$. Moreover, one can choose the constant $c$ so that  $v-c\in \C_{0}(\R; \L^{2}(\R^{n}))$. Thus, taking $G$ to be the zero extension  to $(-\infty,0)\times \R^n$ of $F\in \L^2(\R_{+};  \L^2(\R^n; \IC^n)) $, we see that the zero extension  to $(-\infty,0)\times \R^n$ of $u$ must agree with $v-c $. In particular,  this extension of $u$ is the unique solution in $\L^2(\R; \Hdot^1(\R^n))$ of $\partial_{t}v-\div (A \nabla v)=\div G$ on $\R \times \R^n$ that vanishes on $\{0\}\times \R^n$, and the formulas \eqref{eq:representation} and \eqref{eq:weaksense} are valid for all $t\in \R$ provided one integrates from $-\infty$. We shall use freely this extension principle without changing notation (precisely, the Lions operator can be defined on $\L^2(\R; \L^2(\R^n;\IC^n)$ and it commutes with such zero extensions).

Remark that we have defined $u$ and not only $\nabla u (=\cL F)$. Formally, one is tempted to write 
\begin{align*}
 u(t,x)&= \int _{0} ^{t} \Gamma(t,s)\div F(s,.)(x)\, \d s,
\\
\cL(F)(t,x) &= \int _{0} ^{t} \nabla \Gamma(t,s)\div F(s,.)(x)\, \d s.
\end{align*}
This is a good rule of thumb but the integrals do not converge as such. 
Given $T>0$, the time truncated and rescaled solution operator is defined by 
$$
\left.\begin{array}{cccc}\widetilde\cL_{[0,T]}: & \L^2(\R_{+};  \L^2(\R^n; \IC^n)) & \rightarrow & \L^2(\R_{+};  \L^{2}(\R^n)) \\ & F & \mapsto & [(t,x)\mapsto \frac{1_{[0,T]}(t)}{T^{1/2}} u(t,x)]\end{array}\right.
$$
   that is,
$$
\widetilde\cL_{[0,T]}(F):(t,x) \mapsto \frac{1_{[0,T]}(t)}{T^{1/2}}\bigg(\int _{0} ^{t} \Gamma(t,s)\div F(s,.)\, \d s\bigg)(x).
$$
 The estimates above show that $\widetilde\cL_{[0,T]}$ is bounded, uniformly with respect to $T$ thanks to  the normalisation.

\section{Maximal regularity for the Lions operator in tent spaces}
\label{sec:main}

 Recall that the tent space $\T^{p}$, $0<p<\infty$, is defined by \eqref{def:Tp2} for complex-valued functions.
The spaces are Banach when $p\ge 1$ and quasi-Banach when $p<1$. If $p=\infty$, then $\T^{\infty}$ is the (Banach) space of measurable functions such that 
\begin{equation}\label{def:Tinfty2}
\|g\|_{\T^{\infty}} := \sup_{x\in \R^n,r>0}
\bigg( \int_{0}^{r^2} \fint_{B(x,{r})} 
 |g(t,y)|^{2} \,{\d y \,\d t}\bigg)^{\frac{1}{2}}<\infty \end{equation}

It is known that the spaces $\T^{p}$ interpolate by the complex method for $ 1 \leq p\le \infty$ and the real method for $0<p\leq \infty$;  see \cite{cms} and \cite{long}. 

 If $\H$ is a Hilbert space, a $\H$-valued function $g$, its norm in $\H$ being denoted by  $|g|_{\H}$,  belongs to $\T^{p}(\H)$ if and only if $|g|_{\H}$ belongs to $\T^{p}$ and  $\|g\|_{\T^{p}(\H)}= \||g|_{\H}\|_{\T^{p}}$.

By Fubini's theorem, $\T^{2}$ agrees with $\L^2(\R_{+}\times \R^n)$  isometrically, and the dual of $\T^p$, $1\le p<\infty$, is $\T^{p'}$, $p'$ being the H\"older conjugate exponent to $p$,   for the $\L^2(\R_{+}\times \R^n)$ duality, see \cite{cms}. 

As $\T^2=\L^2(\R_{+}\times \R^n)=\L^2(\R_{+};\L^2(\R^n))$ isometrically, the Lions operator $\cL$ and the time truncated and rescaled solution operator $\widetilde\cL_{[0,T]}$ are defined and bounded  from 
$\T^{2}(\IC^n)$  into $\T^{2}(\IC^d)$, $d=n,1$, respectively,  uniformly in $T$ for the second one.

\begin{thm}
\label{thm:main1}
For $2< p \le \infty$, and $T>0$, $\cL$ and $\widetilde\cL_{[0,T]}$ extend to bounded operators from 
$\T^{p}(\IC^n)$  into $\T^{p}(\IC^d)$, $d=n,1$, respectively. 
\end{thm}

\begin{thm}
\label{thm:main2}
 Assume that $A$ has real-valued coefficients. There exists $\varepsilon>0$ such that if   $1-\varepsilon<p<2$, and $T>0$,  then $\cL$  and $\widetilde\cL_{[0,T]}$ extend to bounded operators  from 
$\T^{p}(\IC^n)$  into $\T^{p}(\IC^d)$, $d=n,1$, respectively. \end{thm}

\begin{proof}[Proof of Theorem \ref{thm:main1}]  
We prove the result for $\cL$. The proof for $\widetilde\cL_{[0,T]}$ is similar.  By complex interpolation between $\T^2$ and $\T^\infty$ and density of $\T^2\cap \T^\infty$ in $\T^p$ when $2<p<\infty$, it suffices to consider the case $p=\infty$ as we already have the case $p=2$. We do it by a direct method which yields the construction  of the extension on $\T^\infty$. It cannot be by density as $\T^2\cap \T^\infty$ is not dense in $\T^\infty$. We can not use  duality with $\T^1$ and weak star topology  either.  Indeed, under our assumption here, we do not know the boundedness on $\T^1$ of the adjoint operator. 

   For a ball  $B=B(x,r)$ of $\R^n$, 
 write $T_{j}(B)= (0, 4^jr^2)\times 2^jB$ for $j$ integer, $j\ge 0$, and set
 $C_{2}(B)=T_{2}(B)$ and $C_{j}(B)= T_{j}(B) \setminus T_{j-1}(B)$ if $j\ge 3$. 
 
 We let
$F \in \T^{\infty}(\IC^n)$ and extend $F$ by 0 in the lower half-space. Define $F_{j}=F$ on $C_{j}(B)$ and 0 otherwise, so that each $F_{j}\in \L^2(\R; \L^2(\R^n;\IC^n))$ and pick $u_{j}\in \L^2(\R; \Hdot^1(\R^n))\cap \C_{0}(\R;  \L^2(\R^n))$  defined from Theorem \ref{thm:CP} and  the extension principle,  so that  $u_{j}=0$ on the lower half-space with  $\nabla u_{j}=\cL F_{j}$ on the upper-half space.   As $F=F_{2}+ F_{3}+ \cdots$, we let 
 $u=u_{2}+ u_{3}+ \cdots$. We shall show  the following facts. 
 \begin{enumerate}
 \item For   some constant $C$  independent of $B$ and $F$ 
 \begin{equation}
\label{eq:CM}
\int_{0}^{r^2} \int_{B} 
 |\nabla u(t,y)|^{2} \,{\d y\, \d t} \le C r^n  \|F\|_{\T^{\infty}(\IC^n)}^2,
 \end{equation}
 where we  use the notation $|\nabla u|$ for $|\nabla u|_{\IC^n}$. 
\item The series defining $u$ converges in $\L^2(a,b; \H^1( \mathcal{O}))\cap \L^\infty(a,b; \L^2( \mathcal{O} ))$ for any bounded interval $[a,b]$ and bounded open set  $ \mathcal{O} \subset \R^n$.
\item $u$ is a weak solution of  $\partial_{t}u-\div (A \nabla u)=\div F$ in any bounded open region $(a,b)\times  \mathcal{O} $ of $\ree$, that is identically 0 in the lower half-space. 
\item $u$ does not depend on the choice of $B$. 
\item If $F\in \T^{2}(\IC^n)\cap \T^{\infty}(\IC^n)$, then $\nabla u$ agrees with $\cL F$. 
\end{enumerate}

Assuming the above items hold, we have built the extension of $\cL$ initially defined on  $\T^{2}(\IC^n)$ to
$\T^{\infty}(\IC^n)$ and this extension is bounded. 

Next, we prove the five items successively. 

\medskip

\paragraph{\itshape Proof of (i)}  As $F_{j}\in \L^2(\R_{+}; \L^2(\R^n; \IC^n))$, we let $u_{j}$ be the weak solution of  \eqref{eq:CP} with right hand-side $\div F_{j}$ and zero initial data provided by Theorem \ref{thm:CP}.  We adopt the convention that $u_{j}=0$ on the lower half space. 
 
 For $j=2$, we use the boundedness of $\cL$, so that 
\begin{equation}
\label{eq:u2}
\int_{0}^{r^2}\int_{B} |\nabla	 u_{2}|^2\, \d x\,\d t   \le \lambda^{-1}\|F_{2}\|_{\L^2(\R_{+}, \L^2(\R^n; \IC^n))}^2 \le  \lambda^{-1} (4r)^n  \|F\|_{\T^{\infty}(\IC^n)}^2.
\end{equation}
 
 For $j\ge 3$, by Caccioppoli's inequality  \eqref{eq:caccio} as $F_{j}$ vanishes on a neighborhood of $(0,r^2)\times 2B$ (and using  $u_{j}=0$ for $t<0$), 
 $$
 \int_{0}^{r^2}\int_{B} |\nabla	 u_{j}|^2\, \d x\,\d t  \le Cr^{-2} \int_{0}^{r^2}\int_{2B} |u_{j}|^2\, \d x\,\d t.
 $$
 We fix  $t\in (0,r^2)$ and let $\varphi\in C_{0}^\infty(2B)$. Then, by \eqref{eq:weaksense}, 
\begin{equation}\label{eq:uj}
\angle{u_{j}(t)}{\varphi}= - \int_{0}^{t}\int_{\R^n} F_{j}(s,y)\cdot \nabla \tilde \Gamma (s,t)\varphi(y)\, \d y\, \d s.
\end{equation}
Set $\tilde v_{t}(s,y)=\tilde \Gamma (s,t)\varphi(y)$, which is a solution of the backward equation $-\partial_{s}v-\div (A^*\nabla v)=0$ on $(-\infty,t)\times \R^n$ with final data $\varphi$ at time $t$. By Cauchy-Schwarz inequality, and using $t\le 4^{j-1}r^2$, it follows that 
\begin{equation}
\label{eq:phiB1}
|\angle{u_{j}(t)}{\varphi}|^2 \le \int_{0}^t\int_{2^{j}B\setminus 2^{j-1}B} |F(s,y)|^2\, \d y\, \d s \int_{0}^t\int_{2^{j}B\setminus 2^{j-1}B} |\nabla \tilde v_{t}(s,y)|^2\, \d y\, \d s.
\end{equation}
The first integral on the right-hand side is bounded by $(2^jr)^n \|F\|_{\T^{\infty}(\IC^n)}^2.$
Using a bounded covering of the range of integration by Whitney balls $(t-2\delta^2, t-\delta^2)\times B_{\delta}$ for the region $(-\infty,t)\times \R^n$, applying Caccioppoli's  inequality \eqref{eq:caccio} on each and summing up we obtain
$$
\int_{0}^t\int_{2^{j}B\setminus 2^{j-1}B} |\nabla \tilde v_{t}(s,y)|^2\, \d y\, \d s \lesssim \int_{0}^t\int_{2^{j+1/2}B\setminus 2^{j-3/2}B} \frac{| \tilde v_{t}(s,y)|^2}{t-s}\, \d y\, \d s.
$$
Since $2^{j+1/2}B\setminus 2^{j-3/2}B$ is at distance  $c2^jr$ from 
 the support $2B$ of $\varphi$, we obtain from the $\L^2-\L^2$ off-diagonal bounds  of $\tilde \Gamma (s,t)$ derived by duality from \eqref{eq:od}  that 
\begin{equation}
\label{eq:phiB2}
\int_{0}^t\int_{2^{j+1/2}B\setminus 2^{j-3/2}B} \frac{| \tilde v_{t}(s,y)|^2}{t-s}\, \d y\, \d s \lesssim  \int_{0}^t  \frac{e^{-c4^jr^2/(t-s)}}{t-s}\, \d s \ \|\varphi\|_{\L^2}^2 \lesssim  e^{-c4^jr^2/t} \ \|\varphi\|_{\L^2}^2.
\end{equation}
Thus,
$$
\|u_{j}(t)\|_{\L^2(2B)}\lesssim (2^jr)^{n/2} \|F\|_{\T^{\infty}(\IC^n)} e^{-c4^jr^2/2t} 
$$
so that 
$$
\int_{0}^{r^2}\int_{2B} |u_{j}|^2\, \d x\, \d t \lesssim (2^jr)^n \|F\|_{\T^{\infty}(\IC^n)}^2  \int_{0}^{r^2} e^{-c4^jr^2/t}\, \d t \lesssim (2^jr)^n \|F\|_{\T^{\infty}(\IC^n)}^2 4^{-j} r^2 e^{-c4^j}
$$ and we obtain 
$$\int_{0}^{r^2}\int_{B} |\nabla	 u_{j}|^2\, \d x\, \d t \lesssim 2^{j{(n-2)}}e^{-c4^j} r^n  \|F\|_{\T^{\infty}(\IC^n)}^2.   
$$
Using Minkowski inequality and summing all the contributions, we obtain
\eqref{eq:CM}. 

\medskip

\paragraph{\itshape Proof of (ii)} As $u=0$ when $t<0$, it  suffices to prove the desired convergence  on $[0,4^k r^2]\times 2^k B$ for  all $k\in \IN$. If we denote $F_{j,B}=F_{j}$ and $u_{j,B}=u_{j}$ the terms defined in (i), then $F_{j,B}= F_{j-k,2^kB}$ and $u_{j,B}= u_{j-k,2^kB}$ for $j-k\ge 3$, and  $F_{2,B}+\cdots + F_{k+2,B}=F_{2,2^kB}$ and $u_{2,B}+\cdots + u_{k+2,B}=u_{2,2^kB}$. The first equality is trivial while the latter follows from uniqueness of the Cauchy problem in Theorem \ref{thm:CP}. 
Thus the series for $\nabla u$ converges in $\L^2((0,4^k r^2)\times 2^k B)$ with bound 
$$
\int_{0}^{4^kr^2}\int_{2^kB} |\nabla	 u|^2\, \d x\, \d t \lesssim (2^k r)^n  \|F\|_{\T^{\infty}(\IC^n)}^2. 
$$
Let us do the $\L^\infty(0,4^k r^2;\L^2(2^k B))$ bound for $u$.  By the same trick, it suffices to do it for $k=0$ with the same decomposition.    But we have shown, if $j\ge 3$ and $0\le t\le r^2$ 
$$
\|u_{j,B}(t)\|_{\L^2(B)}\lesssim (2^jr)^{n/2} \|F\|_{\T^{\infty}(\IC^n)} e^{-c4^jr^2/2t} \le 2^{jn/2}e^{-c4^j/2} r^{n/2} \|F\|_{\T^{\infty}(\IC^n)}
$$
while for $j=2$,  by  \eqref{eq:enestF}, for all $t>0$, 
$$
\|u_{2,B}(t)\|_{\L^{2}} \le \frac 1{\sqrt  {2\lambda}} \|F_{2,B}\|_{\L^2(\R_{+}; \L^2(\R^n,\IC^n))} \le  \frac {1}{ \sqrt{2\lambda}} (4r)^{n/2}  \|F\|_{\T^{\infty}(\IC^n)}.
$$
The convergence of the series for $u$ in $\L^\infty(0, r^2; \L^2(B))$ 
follows. 

\medskip

\paragraph{\itshape Proof of (iii)} Integrating against a test function $\varphi$ with, say, bounded support in $[-4^k r^2,4^k r^2]\times 2^k B$ for some $k\ge 0$, then the 
series $F=F_{2,2^kB}+ F_{3,2^kB}+ \cdots$ reduces to the first term on the support of  $\varphi$. We have by definition  that $u_{j, 2^kB}$ is a weak solution of   $\partial_{t}u_{j, 2^kB}-\div (A \nabla u_{j,2^kB})=\div F_{j,2^kB}$ in $\ree$ for $j\ge 2$. Thus,   
$\partial_{t}u_{2,2^kB}-\div (A \nabla u_{2,2^kB})=\div F$ and 
$\partial_{t}u_{j,2^kB}-\div (A \nabla u_{j,2^kB})=0$ when $j\ge 3$ in $(-4^k r^2,4^k r^2)\times 2^k B$. The convergence established in (ii) allows one to add up these equations in  $(-4^k r^2,4^k r^2)\times 2^k B$ so that  $u$ is a weak solution of  $\partial_{t}u-\div (A \nabla u)=\div F$ in $(-4^k r^2,4^k r^2)\times  2^k B$. 

\medskip

\paragraph{\itshape Proof of (iv)} If we start from another ball $B'$, decompose $F=F_{2,B'}+F_{3,B'}+\cdots$  and create $u_{B'}$ as we created $u=u_{B}$.  Let $B_{\rho}$ be a ball of radius $\rho>0$. We shall show that $u_{B}=u_{B'}$ in $\L^\infty(0, \rho^2; \L^2(B_{\rho}))$. This will show that $u_{B}=u_{B'}$ on $\ree$.  Write
$u_{B}= U_{j,B}+ R_{j,B}$ where for $j\ge 2$, $R_{j,B}$ is the tail with the sum starting at $j+1$. Pick $k\in \IN$ such that $B_{\rho} \subset 2^{k}B\cap 2^{k}B'$. Then observe that for $j\ge 0$, $2^jB_{\rho}\subset 2^{j+k}B\cap 2^{j+k}B'$. Write for $j\ge 2$, 
$u_{B}-u_{B'}= (U_{j+k,B}-U_{j+k,B'})+ R_{j+k,B}-R_{j+k,B'}$. 
From the previous estimates, as $j\to \infty$,   we have that $R_{j+k,B}$ converges to zero in 
$\L^\infty(0, 4^k r^2; \L^2(2^kB))$ and $R_{j+k,B'}$ converges to zero in 
$\L^\infty(0, 4^k{r'}^2; \L^2(2^kB'))$. By uniqueness, $U_{j+k,B}-U_{j+k,B'}$ is the weak solution in $\L^2(\R_{+}; \Hdot^1(\R^n))$ of the parabolic equation with right-hand side $\div G_{j}$ where $G_{j}= F(1_{T_{j+k}(B)}- 1_{T_{j+k}(B')})$ and zero initial data. Note that $G_{j}$  is uniformly bounded in  $\T^{\infty}(\IC^n)$ with respect to $j$ and is zero at least in $T_{j}(B_{\rho})$. Hence, performing the same decomposition  as above according to the ball $B_{\rho}$, we have that $U_{j+k,B}-U_{j+k,B'}$ converges to 0 in 
$\L^\infty(0, \rho^2; \L^2(B_{\rho}))$ as $j\to \infty$. 

\medskip

\paragraph{\itshape Proof of (v)} If $F$ also belongs to $\T^{2}(\IC^n)$, the convergence of the series used to decompose $F$ given a ball $B$ is also in $\T^{2}(\IC^n)$, hence $\cL F$ is the gradient of the series to define $u_{B}$, that is, $\nabla u_{B}$ coincides with $\cL F$. 
 \end{proof}

\begin{proof}[Proof of Theorem \ref{thm:main2}]  
 Again we just prove the result for $\cL$ as the proof for $\widetilde\cL_{[0,T]}$ is similar.  
By complex interpolation between $\T^{1}$ and $\T^{2}$ and density of $\T^1\cap \T^2$ in $\T^p$ when $1<p<\infty$, it suffices to establish the required bound when $p\le 1$. In this case,  atomic decomposition of tent spaces implies that a bounded linear operator $T$ on $\T^{2}(\IC^n)$ extends to a bounded operator on  $\T^{p}(\IC^n)$  if it is uniformly bounded on $\T^{p}(\IC^n)$ atoms (see Step 3 of the proof of \cite[Thm 4.9]{amr}, which can be easily adapted for $p<1$),  that is, for all 
$a:(0,\infty)\times \R^{n}\to \IC^n$, with support in a region $(0,r^2)\times B_{r}$, where $B_{r}$ is a  ball  with radius $r$, and estimate
\begin{equation}
\int_{0}^{r^2}\int_{B_{r}} |a(t,x)|_{\IC^n}^2 \, \d x \, \d t \le r^{n(1-2/p)},
\end{equation}
we have for $C$ independent of $a$,
\begin{equation}
\label{eq:TboundedonTpatom}
\|T a\|_{\T^{p}(\IC^n)} \le C.
\end{equation}

We show that there exists $\varepsilon>0$ so that one can establish \eqref{eq:TboundedonTpatom} for $\cL$ and $1-\varepsilon<p\le 1$. 
 
By translation and dilation invariance of the hypotheses on the coefficients $A$, we may assume $r=1$ and $B=B_{1}$ is the unit ball. As mentioned,  we extend $a$ by 0 on the lower-half space and from now on consider $\cL a=\nabla u$ with $u$ a weak solution of $\partial_{t}u-\div (A \nabla u)=\div a$ on $\R \times \R^n$ (recall that $A$ is defined on all of $\R \times \R^n$) which vanishes on the lower-half space. 

As $a\in \L^2(\ree)$, we know that $u\in \C_{0}(\R; L^2(\R^n))$ and is given   for each $t\in \R$ by $u(t)=  \int_{-\infty}^t\Gamma(t,s)(\div a(s))\, \d s$ using \eqref{eq:representation}. 

We write $\nabla u=\sum_{j\ge 2} F_{j}$ where $F_{j}=1_{C_{j}(B)}\nabla u$ and we show that 
$$\|F_{j}\|_{\L^2(\R_{+};\L^2(\R^n;\C^n))} \le C_{n,\lambda,\Lambda} 2^{-j(n+2\eta)/2}$$ for some $\eta>0$. This implies that we can write $F_{j}=\lambda_{j}a_{j}$ where $a_{j}$ is a $\T^{p}(\IC^n)$ atom, associated to the region $T_{j}(B)$, and $(\lambda_{j})\in \ell^p$ as long as $1\ge p>\frac n{n+\eta}$, so that the series for $\nabla u$ converges in $\T^{p}(\IC^n)$. \\
We begin with $F_{2}$. We have
$$
\|F_{2}\|_{\L^2(\R_{+};\L^2(\R^n;\C^n))}   \le  \|\nabla u\|_{\L^2(\R_{+};\L^2(\R^n;\C^n))}  \le \lambda^{-1}\|a\|_{\L^2(\R_{+};\L^2(\R^n;\C^n))}.
$$

We decompose further $F_{j}=F_{j}^1+F_{j}^2$ when $j\ge 3$ with 
\begin{align*}
 F_{j}^1&=1_{(0,4^{j-1})}(t)F_{j}= 1_{(0,4^{j-1})\times 2^{j}B\setminus 2^{j-1}B} \nabla u
 \\
 F_{j}^2&= 1_{(4^{j-1}, 4^j)}(t)F_{j}= 1_{(4^{j-1}, 4^j)\times 2^jB}\nabla u.
\end{align*}

We look at $F_{j}^1$ first. Using Caccioppoli's inequality \eqref{eq:caccio} and $u=0$ for $t<0$,   
$$   
\int_{0}^{4^{j-1}}\int_{2^{j}B\setminus 2^{j-1}B} |\nabla	 u|^2\, \d x\,\d t  \le C4^{-j} \int_{0}^{4^{j-1}}\int_{2^{j+1/2}B\setminus 2^{j-3/2}B} |u|^2\, \d x\,\d t . 
$$
 We let $\varphi\in \C_{0}^\infty(2^{j+1/2}B\setminus 2^{j-3/2}B)$. Then we have the representation formula:
\begin{equation}
\label{eq:rep}
\angle{u(t,.)}{\varphi}= - \int_{0}^{t\wedge 1}\int_{B} a(s,y)\cdot \nabla \tilde \Gamma (s,t)\varphi(y)\, \d y\, \d s.
\end{equation}
We begin with  estimating the integral for $t\in (0,32)$. 
Set $\tilde v_{t}(s,y)=\tilde \Gamma (s,t)\varphi(y)$, which is a solution of the backward equation $-\partial_{s}v-\div (A^*\nabla v)=0$ on $(-\infty,t)\times \R^n$ with final data $\varphi$ at time $t$. It follows that 
$$
|\angle{u(t,.)}{\varphi}|^2 \le \int_{0}^t\int_{B} |\nabla \tilde v_{t}(s,y)|^2\, \d y\,\d s.
$$
Using bounded covering of the range of integration by Whitney balls $(t-2\delta^2, t-\delta^2)\times B_{\delta}$ with $B_{\delta}$ balls of radius $\delta$,  applying Caccioppoli's  inequality \eqref{eq:caccio} for backwards solutions on each Whitney ball and summing up, we obtain
$$
\int_{0}^t\int_{B} |\nabla \tilde v_{t}(s,y)|^2\, \d y\,\d s \lesssim \int_{0}^t\int_{2B} \frac{| \tilde v_{t}(s,y)|^2}{t-s}\, \d y\,\d s.
$$
Since the support of $\varphi$ is at a distance at least $c2^j$ from $2B$, we obtain from the $\L^2-\L^2$ off-diagonal bounds   \eqref{eq:od}   for   $\tilde \Gamma (s,t)$ that 
$$
\int_{0}^t\int_{2B} \frac{| \tilde v_{t}(s,y)|^2}{t-s}\, \d y\, \d s \lesssim  \int_{0}^t  \frac{e^{-c4^j/(t-s)}}{t-s}\, \d s \ \|\varphi\|_{\L^2}^2 \lesssim  e^{-c4^j/t} \ \|\varphi\|_{\L^2}^2.
$$
Thus,
$$
\int_{0}^{32}\int_{2^{j+1/2}B\setminus 2^{j-3/2}B} |u(s,y)|^2\, \d y \, \d s \lesssim  \int_{0}^{32} e^{-c4^j/t}\, \d t \lesssim  e^{-c4^j/32}.
$$
 Now we assume $32\le t$ and $t \le 4^{j-1}$. With $\varphi$ as before, we restart from 
 $$
|\angle{u(t,.)}{\varphi}|^2 \le \int_{0}^1\int_{B} |\nabla \tilde v_{t}(s,y)|^2\, \d y\, \d s.
$$
Write 
$\nabla  \tilde v_{t}(s,y)=\nabla (\tilde v_{t}(s,y)-\tilde v_{t}(0,0))$. As $\tilde v_{t}-\tilde v_{t}(0,0)$ is also a solution of the backward equation, by the Caccioppoli inequality directly on $(0,1)\times B$ as $t-s\ge 31$,  
$$
|\angle{u(t,.)}{\varphi}|^2 \lesssim \int_{0}^{2}\int_{2B} |\tilde v_{t}(s,y)-\tilde v_{t}(0,0)|^2\, \d y\, \d s.
$$
 Using that the coefficients are real, we can invoke Nash's estimate  (see the statement of Theorem C in \cite{aronson}).  This gives us the existence of some $\eta>0$ such 
that, for all $s\in [0,2] $ and $y\in 4B$, one has  $$
|\tilde v_{t}(s,y)-\tilde v_{t}(0,0)| \le C\bigg(\frac{|s|^{1/2}+|y|}{t^{1/2}}\bigg)^\eta \sup_{[-t/8, t/8]\times \overline B(0,\sqrt t/2)} |\tilde v_{t}|,
$$
as  $ 2\le t/16$ and $4B\subset B(0,\sqrt t/2)$.
By the pointwise Gaussian decay of the kernel of $\tilde \Gamma (s,t)$ (see \cite{aronson} or \cite{ae}) and the support condition of $\varphi$, we have for $t\le 4^{j-1}$,
$$
\sup_{[-t/8, t/8]\times \overline B(0,\sqrt t/2)} |\tilde v_{t}|\le Ct^{-n/2}e^{-c4^j/t}\|\varphi\|_{\L^1}.
$$
Thus we obtain
\begin{align*}
 |\angle{u(t,.)}{\varphi}|^2 \lesssim  \int_{0}^{2} |2B| \frac{e^{-2c4^j/t}}{t^{n+\eta}}  \ \|\varphi\|_{\L^1}^2\, \d s
\end{align*}
hence, for $x\in  2^{j+1/2}B\setminus 2^{j-3/2}B$,
$$
|u(t,x)|^2 \lesssim \frac{e^{-2c4^j/t}}{t^{n+\eta}}.
$$
It follows that $\|F_{j}^1\|_{\L^2(\R_{+};  \L^2(\R^n; \IC^n))}^2$  is controlled by 
$$
 4^{-j}e^{-c4^j/32}  +4^{-j} \int_{32}^{4^{j-1}}\int_{2^{j+1/2}B\setminus 2^{j-3/2}B} \frac{e^{-2c4^j/t}}{t^{n+\eta}}\, \d x\, \d t
\lesssim  2^{-j(n+2\eta)}.
$$

We turn to $F_{j}^2$. We start again by using Caccioppoli's inequality
$$ 
\int_{4^{j-1}}^{4^{j}}\int_{2^{j}B} |\nabla	 u|^2\, \d x\,\d t   \le C4^{-j} \int_{4^{j-3/2}}^{4^{j}}\int_{2^{j+1}B} |u|^2\, \d x\,\d t  . 
$$
  We let $\varphi\in \C_{c}^\infty(2^{j+1}B)$. Then as $t>1$, 
$$
\angle{u(t,.)}{\varphi}= -  \int_{0}^{1}\int_{B} a(s,y)\cdot \nabla \tilde \Gamma (s,t)\varphi(y)\, \d y\, \d s.
$$
so that
$$
|\angle{u(t,.)}{\varphi}|^2 \le \int_{0}^1 \int_{B} |\nabla \tilde v_{t}(s,y)|^2\, \d y\,\d s,
$$
with $\tilde v_{t}(s,y)=\tilde \Gamma (s,t)\varphi(y)$. Again, we may replace $\tilde v_{t} $ by $\tilde v_{t}-\tilde v_{t}(0,0)$ and use the Caccioppoli inequality directly on $(0,1)\times B$ as $t \ge 4^{j-1}$  to obtain 
$$
|\angle{u(t,.)}{\varphi}|^2 \lesssim \int_{0}^{2}\int_{2B} |\tilde v_{t}(s,y)-\tilde v_{t}(0,0)|^2\, \d y\,\d s.
$$
Observe that 
$\sup_{[-t/8, t/8]\times \overline B(0,\sqrt t/2)} |\tilde v_{t}| \lesssim t^{-n/2}\|\varphi\|_{1}$, so that  the Nash inequality again yields
$$
|\angle{u(t,.)}{\varphi}|^2 \lesssim t^{-(n+\eta)}\|\varphi\|_{\L^1}^2
$$
and we conclude that 
$$
\|F_{j}^2\|_{\L^2(\R_{+};  \L^2(\R^n; \IC^n))}^2 \lesssim 4^{-j(n+\eta)}|2^{j+1}B| \eqsim 2^{-j(n+2\eta)}.
$$

\end{proof}
\begin{rem} In absence of the Nash inequality, a variant of the argument yields $2^{-jn} $ for the last upper bound, which does not allow  us  to sum the estimates in $j$ in $\T^{p}(\IC^n)$ for $p\le 1$. We have not found how to improve this bound without the Nash inequality. 
 \end{rem}

\begin{rem} The proof of Theorem~\ref{thm:main2} applies to any equation for which  Nash inequality and the Gaussian pointwise upper bound hold. So having real coefficients in $A$ is not necessary. Hofmann and Kim \cite{HK} show that this holds for $\L^\infty$-small complex perturbations of 
real coefficients. This is true also in dimension $n=1$ and $n=2$ for coefficients  close enough in $\L^\infty$  to  a $t$-independent elliptic matrix (and this would work for systems).  
 
\end{rem}

\begin{rem}
\label{rem:CP} Let $1<p<\infty$ and assume that $A$ has real entries. 
 Combining these two results with \cite[Section 7]{amp}, where \eqref{eq:LPestA} is established,    the uniqueness result of  \cite[Thm. 3.11]{amp} when $p=2$ already mentioned in the proof of  Theorem \ref{thm:CP},  and its extension to  all such $p$, see \cite[Thms.~1.1 and~1.5]{zaton},  this proves well-posedness of global weak solutions for  the Cauchy problem  \eqref{eq:CP} with  $\nabla u \in \T^p(\IC^n)$, for data $(\psi,F)\in \L^p\times \T^p(\IC^n)$, the relation  $u(0)=\psi$ being achieved as  $u(t)\to \psi$, $t\to0$,  in $\L^1_{\loc}$ sense. As the proof of uniqueness will be a special case of our well-posedness result  Theorem \ref{thm:spde}, we refer to Step 1 of its proof for details. 
\end{rem} 

\begin{rem} In the case where  $A$ depends only on the $x$ variable,   more can be said, see \cite{ah}, using semigroup techniques and improvement of the singular integral operator theory of \cite{akmp}.  First, the Lions map $F\mapsto \nabla u$ and the map   $F \mapsto ((t,x)\mapsto t^{-1/2}u(t,x))$ are bounded on $\T^p$ for $p$ in a range containing $[\frac{2n}{n+2},\infty]$ even if $A$ has complex entries (thus, Nash's regularity is not needed). Note that the second operator is larger that $\cL_{[0,T]}$ and no truncation is required. Second, estimates, as well as uniqueness, in weighted tent estimates (with a power weight $t^{-2\beta}$  for $\beta$ in some range including $\beta=0$) can also be established. 
\end{rem}

\section{Quadratic maximal regularity operators in tent spaces}
\label{sec:quadratic}

We consider the  quadratic maximal regularity (sublinear) operators whose key role in the study of the stochastic heat equation will be discussed in Section~\ref{sec:stoch}.  The operator $\cQ$ is defined by \eqref{eq:quadmaxheat}, with a companion defined 
for $0<T<\infty$ by
$$
\left.\begin{array}{cccc}\widetilde\cQ_{[0,T]}: & \L^2(\R_{+};  \L^2(\R^n)) & \rightarrow & \L^2(\R_{+};  \L^2(\R^n)) \\ & f & \mapsto & \left[(t,x) \mapsto \big(\frac{1_{[0,T]}(t)} T \int _{0} ^{t} | e^{(t-s) \Delta}f(s,.)(x)|^{2}\,  \d s\big)^{\frac{1}{2}}\right]\end{array}\right.
$$
These operators are well defined and bounded.  By the square function estimate $$\| (t,x) \mapsto \nabla e^{t\Delta}h(x) \|_{\L^2(\R_{+};  \L^2(\R^n;\IC^n))}^2 = \frac 1 2 \|h\|_{\L^{2}}^2$$
 and Fubini's theorem, one has
\begin{align*}
\|\cQ(f)\|_{\L^2(\R_{+};  \L^2(\R^n))} ^{2} 
& =  \int _{0} ^{\infty} \int _{s} ^{\infty} \int _{\R^{n}} |\nabla e^{(t-s) \Delta}f(s,.)(x)|^{2}_{\IC^n}\,  \d x\,\d t\, \d s \\
& = \frac 1 2  \int _{0} ^{\infty} \int _{\R^{n}}  |f(s,x)|^{2}\,  \d x\, \d s 
\\
&= \frac 1 2 \|f\|_{\L^2(\R_{+};  \L^2(\R^n))} ^{2}, 
\end{align*}
while the contractivity of the heat semigroup implies
\begin{align*}
\|\widetilde\cQ_{[0,T]}(f)\|_{\L^2(\R_{+};  \L^2(\R^n))} ^{2} 
& = \frac 1 T   \int _{0} ^{T} \int _{s} ^{T} \int _{\R^{n}} | e^{(t-s) \Delta}f(s,.)(x)|^{2}\, \d x\,\d t\, \d s \\
& \leq   \int _{0} ^{T} \int _{\R^{n}}  |f(s,x)|^{2}\, \d x\, \d s 
\\
&\leq  \|f\|_{\L^2(\R_{+};  \L^2(\R^n))} ^{2}.
\end{align*}

\begin{rem}  Replacing $\frac{1_{[0,T]}(t)} T$ by the larger  $\frac{1_{[0,T]}(t)} t$ in the definition of $\widetilde\cQ_{[0,T]}$ leads to an unbounded operator. 
 \end{rem}

\begin{thm}
\label{thm:quad1}
The quadratic maximal regularity operators $\cQ$ and $\widetilde\cQ_{[0,T]}$ extend to  bounded sublinear operators from $\T^{p}$ to $\T^{p}$ for all $p \in [2,\infty]$.
\end{thm}
\begin{proof} 
We argue first for $\cQ$. 
For $p=\infty$, we can follow the proof of Theorem \ref{thm:main1}.
We indicate the required changes. 
We let $f \in \T^{\infty}$
and, using the same notation as in the proof of Theorem \ref{thm:main1}, we define $f_{j}(t,x):=1_{C_{j}(B)}(x)f(t,x)$ for $j=2,...,\infty$, and $B$ a fixed ball of radius $r$. We need to estimate
 $$
 I_{j}:=\int_{0}^{r^2}\int_{B} \int _{0} ^{t} |\nabla e^{(t-s)\Delta}f_{j}(s,.)(x)|^{2} \,\d s\, \d x\, \d t.
$$
For $j=2$, this follows from the $p=2$ case as in the proof of Theorem \ref{thm:main1}.
We now consider $j \geq 3$.
Using Fubini's theorem and a change of variable we have 
$$
I_{j}\le \int_{0}^{r^2}\!\!\int_{B} \int _{0} ^{r^2} |\nabla e^{t\Delta}f_{j}(s,.)(x)|^{2}\, \d t\, \d x\,\d s.$$
Using the explicit formula for the kernel of  $\sqrt t \nabla e^{t\Delta}$, we have  the following  Gaussian decay: 
\begin{equation}
\label{eq:quadmaxinfty}
\int_{B}  | \nabla e^{t\Delta}f_{j}(s,.)(x)|^{2}\, \d x 
\lesssim  t^{-1}{e^{-c\frac{4^{j}r^{2}}{t}}}\|f_{j}(s,.)\|_{\L^2} ^{2}.
\end{equation}
Integrating in $t,s$, we easily have that 
$$
I_{j}\lesssim e^{-c'4^j}   \int_{0}^{r^2}  \|f_{j}(s,.)\|_{\L^{2}} ^{2} \, \d s  \lesssim (4^jr)^{n}
e^{-c'4^{j}} \|f\|_{\T^{\infty}} ^{2},
$$
Summing in $j$ using Minkowski inequality gives that 
$$
\bigg(\int_{0}^{r^2}\int_{B} \int _{0} ^{t} |\nabla e^{(t-s)\Delta}f(s,.)(x)|^{2}\, \d s\, \d x\,\d t\bigg)^{1/2}  \lesssim r^{n/2}\|f\|_{\T^{\infty}}.$$ 
The rest of the proof that yields a proper definition  of $\cQ(f)$ on $\T^{\infty}$ is then the same as in the proof of Theorem \ref{thm:main1}. For $2<p<\infty$, we can proceed by real interpolation and the extension to all of $\T^{p}$ can be achieved by 
 density of $C_{c}^{\infty}(\R_{+}\times \R^{n}) \subset \T^{p}\cap \T^{\infty}$ in $\T^{p}$.

For $\widetilde \cQ_{[0,T]}$, one replaces $\nabla $ by $1/T^{1/2}$ with $t\le T$ and obtains a bound smaller than the one in  the right hand side of \eqref{eq:quadmaxinfty} for the corresponding integral. 
\end{proof}

For $p<2$, we use another approach, based on the following Calder\'on-Zygmund decomposition of tent space functions, discovered by Huang in his PhD thesis (and published in \cite{yi}).

\begin{lem}
\label{lem:cz}
Let $p \in (0,\infty)$. There exists $C>0$ such that, for all $f \in \T^{p}$ and $\lambda>0$, there exist a family of balls $(B_{i})_{i \in \mathbb{N}}$ of $\R^{n}$, and a (almost everywhere) decomposition
$$
f= g + \sum _{i=1} ^{\infty} b_{i},
$$
such that $\text{supp}\, b_{i} \subset \widehat{B_{i}}:= \{(t,y) \in \R_{+} \times \R^{n} \;;\; B(y,\sqrt{t}) \subset B_{i}\}$, and the following estimates hold:
\begin{align}
\|g\|_{\T^{\infty}} &\leq C \lambda, \\
\label{badnorm} \|b_{i}\|_{\T^{p}} &\leq C \lambda |B_{i}|^{\frac{1}{p}} \quad \forall i \in \mathbb{N},\\
\label{badsize} \sum _{i=1} ^{\infty} |B_{i}| &\leq \lambda^{-p} \|f\|_{\T^{p}} ^{p},\\
\label{overlap} \sum _{i=1} ^{\infty} 1_{2B_{i}} &\leq C.
\end{align}
\end{lem}

\begin{thm}
\label{thm:quad2}
The quadratic maximal regularity operators $\cQ$  and $\widetilde \cQ_{[0,T]}$ extend to  bounded sublinear operators from $\T^{p}$ to $\T^{p}$ for all $p \in (1,2]$.
\end{thm}

\begin{proof} We argue with $\cQ$. The proof for $\widetilde \cQ_{[0,T]}$ uses the  same strategy and is left to the reader, taking into account the remarks at the end of the proof of Theorem \ref{thm:quad1}.
 
Let $p \in (1,2)$ and $f \in \T^{p}\cap \T^{2}$. Setting $$\mathcal{A}(f):x \mapsto \left( \int _{0} ^{\infty} \fint _{B(x,\sqrt{t})} |f(t,y)|^{2} \, \d y\, \d t \right)^{\frac{1}{2}}$$ so that $\|f\|_{\T^{p}}=\|\mathcal{A}(f)\|_{\L^{p}}$, it is enough, by real interpolation and 
 density of $\T^{p}\cap \T^{2}$ in $\T^{p}$, 
to prove the weak type estimate
$$
|\{\mathcal{A}(\cQ(f))>3\lambda\}| \lesssim \lambda ^{-p} \|f\|_{\T^{p}}^p, \quad \forall \lambda>0.
$$
We do this    using the Calder\'on-Zygmund decomposition $f = g + \sum _{i=1} ^{\infty} b_{i}$ at height $\lambda$ provided by Lemma \ref{lem:cz}. For the good part $g$, we have that 
\begin{align*}
|\{\mathcal{A}(\cQ(g))>\lambda\}| &\leq \lambda^{-2} \|\cQ(g)\|_{\T^{2}} ^{2}
\lesssim \lambda^{-2} \|g\|_{\T^{2}} ^{2} \lesssim \lambda^{-2} \|g\|_{\T^{\infty}}^{2-p} \|g\|_{\T^{p}}^{p} \\
& \lesssim \lambda^{-p} (\|f\|_{\T^{p}}^{p} + \sum _{i=1} ^{\infty} \|b_{i}\|_{\T^{p}}^{p}) \\
& \lesssim \lambda^{-p} \|f\|_{\T^{p}}^{p} + \sum _{i=1} ^{\infty} |B_{i}| \lesssim \lambda^{-p} \|f\|_{\T^{p}}^{p},
\end{align*}
using the properties of $g$ and $b_{i}$ in Lemma \ref{lem:cz}.
Since the bad components $b_{i}$ do not necessarily have mean $0$, we introduce
$$
b_{i} ^{\sharp}:(t,x) \mapsto 1_{B_{i}}(x) \fint _{B_{i}} b_{i}(t,y)\, \d y,
$$
and let $\widetilde{b}_{i}:= b_{i} - b_{i} ^{\sharp}$, for all $i \in \mathbb{N}$. 
Using \eqref{overlap} first, and \eqref{badsize} second, we have that
\begin{align*}
|\{\mathcal{A}(\cQ(\sum _{i=1} ^{\infty} b_{i}^{\sharp}))>\lambda\}| &\leq \lambda^{-2} \|\cQ(\sum _{i=1} ^{\infty} b_{i}^{\sharp})\|_{T^{2}} ^{2} 
 \\
& \lesssim  \lambda^{-2} \|\sum _{i=1} ^{\infty} b_{i}^{\sharp}\|_{\L^2(\R_{+};  \L^2(\R^n))} ^{2} 
\\
& \lesssim \lambda^{-2} \sum _{i=1} ^{\infty} 
|B_{i}| \bigg(\int _{0} ^{\infty} \bigg|\fint _{B_{i}} b_{i}(t,y)\,\d y\bigg|^{2} \d t\bigg) \\
&\lesssim \lambda^{-2} \underset{i \in \mathbb{N}}{\sup} \bigg(\int _{0} ^{\infty} \bigg|\fint _{B_{i}} b_{i}(t,y)\,\d y\bigg|^{2} \d t\bigg) \lambda^{-p} \|f\|_{\T^{p}}^{p}.
\end{align*}
For $i \in \mathbb{N}$, we let $r_{i}$ denote the radius of $B_{i}$, and note that $b_{i}(t,.)=0$ for all $t>r_{i}^{2}$ and  is supported in $\overline{B_{i}}$ otherwise  since $\text{supp}\, b_{i} \subset \widehat{B_{i}}$.
Using Fubini's theorem, Minkowski inequality, Jensen's inequality, and \eqref{badnorm}, this implies that
\begin{align*}
& \bigg(\int _{0} ^{\infty} \bigg|\fint _{B_{i}} b_{i}(t,y)\, \d y\bigg|^{2} \d t\bigg)^{\frac{1}{2}} = \bigg(\int _{0} ^{r_{i}^{2}} \bigg|\fint _{B_{i}} 
\bigg(\fint _{B(y,\sqrt{t})}\,  \d z\bigg)\,
b_{i}(t,y)\,  \d y\bigg|^{2} \d t \bigg)^{\frac{1}{2}}
\\
&\quad \lesssim \bigg(\int _{0} ^{r_{i}^{2}} \bigg|\fint _{2B_{i}} 
\fint _{B(z,\sqrt{t})}
|b_{i}(t,y)|\, \d y\, \d z\bigg|^{2} \d t \bigg)^{\frac{1}{2}} \\ &\quad \leq \fint _{2B_{i}}\bigg(  \int _{0} ^{r_{i}^{2}}  
\bigg|\fint _{B(z,\sqrt{t})} |b_{i}(t,y)|\, \d y\bigg|^2\,   \d t \bigg)^{\frac{1}{2}} \, \d z
\\
&
\quad  \leq \fint _{2B_{i}} \left(  \int _{0} ^{\infty}  
\fint _{B(z,\sqrt{t})} |b_{i}(t,y)|^{2} \, \d y\,  \d t \right)^{\frac{1}{2}}\, \d z \\
&  \quad \leq \left( \fint _{2B_{i}} \mathcal{A}(b_{i})(z)^{p}\,  \d z \right)^{\frac{1}{p}}
\lesssim |B_{i}|^{-\frac{1}{p}} \|b_{i}\|_{\T^{p}} \lesssim \lambda.
\end{align*}
Combining this with our previous estimate, we have that
$$
|\{\mathcal{A}(\cQ(\sum _{i=1} ^{\infty} b_{i}^{\sharp}))>\lambda\}| \lesssim \lambda^{-p} \|f\|_{\T^{p}}^{p}.
$$

It remains to consider the term with the $ \widetilde{b_{i}}$, which carries the main difficulty.  Remark that as $f$, $g$ and $\sum b_{i}^{\sharp}$ belong to $\T^{2}$, the same holds for 
$\sum \widetilde{b_{i}}$, so $\cQ$ acts properly on it.
Using \eqref{badsize} again, we only have to prove the estimate
\begin{equation}
\label{quadmainest}
|\{x \in O
\;;\; \mathcal{A}(\cQ(\sum_{i=1} ^{\infty} \widetilde{b_{i}}))(x)>\lambda\}| \lesssim \lambda^{-p} \|f\|_{\T^{p}}^{p}.
\end{equation}
where $O=\R^{n} \backslash \cup_{i \in \mathbb{N}}4B_{i}$. 

Let us set
$$
I(x):= \mathcal{A}(
\cQ(\sum _{i=1} ^{\infty} \widetilde{b_{i}}))(x) \quad \forall x 
\in O.
$$
We will prove below the following pointwise estimate, where $c_{i}$ denotes the centre of $B_{i}$,
\begin{equation}
\label{I(x)est}
I(x) \lesssim \lambda \sum _{i=1} ^{\infty} \frac{r_{i}^{n+1}}{|x-c_{i}|^{n+1}}  \quad \forall x  \in O.
\end{equation}
Assuming \eqref{I(x)est}, the proof of  \eqref{quadmainest} can be concluded as follows
\begin{align*}
|\{x \in O
\;;\; I(x) > \lambda\}| 
\lesssim \sum _{i=1} ^{\infty} \int _{O
}  \frac{r_{i}^{n+1}}{|x-c_{i}|^{n+1}} \, \d x \lesssim \sum _{i=1} ^{\infty} r_{i}^{n} \lesssim \frac{ \|f\|_{\T^{p}}^{p}}{\lambda^{p}},
\end{align*}
using Markov's inequality first,  and  \eqref{badsize} in the last inequality.
To prove \eqref{I(x)est}, we exploit the fact that $\widetilde{b_{i}}(t,.)$ has  support in $\overline{B_{i}}$ with   $\fint _{B_{i}} \widetilde{b_{i}}(t,x)\, \d x = 0$ for all $t>0$. Using duality from \cite[Chapter 3, Lemma 7.1]{necas},  or the direct proof from \cite{bog1, bog2}, 
 we have that, for all $i \in \mathbb{N}$ and $t>0$, there exists  $M_{i}(t,.) \in \W^{1,p}(\R^n)$ with support in $\overline{B_{i}}$ (hence   $M_{i}(t,.) \in \W_{0}^{1,p}(B_{i})$)  such that 
\begin{align*}
\widetilde{b_{i}}(t,.) &= \div  M_{i}(t,.)  \quad \& \quad 
\|\nabla M_{i}(t,.)\|_{\L^{p}}  \lesssim \|\widetilde{b_{i}}(t,.)\|_{\L^{p}},
\end{align*}
with implicit constants independent of $t$ and $i$ (remark that the representation in the direct proof explicitly furnishes joint measurability in the $(t,x)$ variables from that of the $\tilde b_{i}$).
The proof of \eqref{I(x)est} is done
in two steps. In the first step, we establish
\begin{equation}
\label{I(x)est1}
I(x) \lesssim  \sum _{i=1} ^{\infty} \frac{r_{i}}{|x-c_{i}|^{n+1}} \bigg(\int _{0} ^{\infty} r_{i}^{-2}\|M_{i}(s,.)\|_{\L^{1}}^{2}\, \d s\bigg)^{\frac{1}{2}} \quad \forall x 
\in O.
\end{equation}
In the second, we show 
\begin{equation}
\label{I(x)est2}
\bigg(\int _{0} ^{\infty} r_{i}^{-2}\|M_{i}(s,.)\|_{\L^{1}}^{2}\, \d s\bigg)^{\frac{1}{2}} \lesssim \lambda |B_{i}| \quad \forall i \in \mathbb{N}.
\end{equation}
Let us start with \eqref{I(x)est1}. Using the kernel estimates for $\nabla e^{(t-s) \Delta}\div$, we have for that, some numerical constant $c>0$,
\begin{align*}
I(x) & \leq \sum _{i=1} ^{\infty} \left( \int _{0} ^{\infty} \fint _{B(x,\sqrt{t})} \int _{0} ^{t}   
 |\nabla e^{(t-s) \Delta}\div M_{i}(s,.)(y)|^{2} \, \d y\, \d s\, \d t \right)^{\frac{1}{2}} \\
&\lesssim \sum _{i=1} ^{\infty} \left( \int _{0} ^{r_{i}^{2}} \int _{s} ^{\infty}   \fint _{B(x,\sqrt{t})} 
\bigg(\int _{B_{i}} \frac{e^{-c\frac{|y-z|^{2}}{t-s}}}{(t-s)^{\frac{n}{2}+1}} |M_{i}(s,z)| \d z\bigg)^{2}
\d y\, \d t\, \d s \right)^{\frac{1}{2}} \\
& =: \sum _{i=1} ^{\infty} (N_{i}(x) + G_{i}(x)),
\end{align*}
where, for $i \in \mathbb{N}$ and $x\in O,$
$$
N_{i}(x):=
\left( \int _{0} ^{r_{i}^{2}} \int _{s} ^{\infty} \fint _{B(x,\sqrt{t})}    
1_{2B_{i}}(y)
\bigg(\int _{B_{i}} \frac{e^{-c\frac{|y-z|^{2}}{t-s}}}{(t-s)^{\frac{n}{2}+1}} |M_{i}(s,z)|\, \d z\bigg)^{2}
\d y\, \d t\, \d s \right)^{\frac{1}{2}}. 
$$
and 
$$
G_{i}(x):=
\left( \int _{0} ^{r_{i}^{2}} \int _{s} ^{\infty} \fint _{B(x,\sqrt{t})}   
1_{(2B_{i})^c}(y)
\bigg(\int _{B_{i}} \frac{e^{-c\frac{|y-z|^{2}}{t-s}}}{(t-s)^{\frac{n}{2}+1}} |M_{i}(s,z)| \,\d z\bigg)^{2}
\d y\, \d t\, \d s \right)^{\frac{1}{2}}. 
$$
We treat $N_{i}$ first. For $y \in 2B_{i}$, $x \not \in 4B_{i}$, $s\leq r_{i}^{2}$, and $|x-y|^{2} \leq t$,
we have that $t \geq \frac{|x-c_{i}|^{2}}{4} \geq 4r_{i}^{2}$ and that
$t-s \geq \frac{3}{4}{t}$.
We thus have
\begin{align*}
N_{i}(x) &\leq \left( \int _{0} ^{r_{i}^{2}} \int _{\frac{|x-c_{i}|^{2}}{4}} ^{\infty}  \fint _{B(x,\sqrt{t})}   
1_{2B_{i}}(y)
(t-s)^{-(n+2)}\|M_{i}(s,.)\|_{\L^{1}} ^{2} \, 
\d y\, \d t\, \d s \right)^{\frac{1}{2}} \\
&\lesssim r_{i}^{n/2}\left( \int _{0} ^{r_{i}^{2}} \int _{\frac{|x-c_{i}|^{2}}{4}} ^{\infty}  
t^{-(3n/2+2)}\|M_{i}(s,.)\|_{\L^{1}} ^{2} \,
\d t\,  \d s \right)^{\frac{1}{2}} \\
&\lesssim \frac{r_{i}^{n/2}}{|x-c_{i}|^{3n/2+1}} \left( \int _{0} ^{r_{i}^{2}}\|M_{i}(s,.)\|_{\L^{1}} ^{2}\, \d s \right)^{\frac{1}{2}} \\
&\lesssim \frac{1}{|x-c_{i}|^{n+1}} \left( \int _{0} ^{r_{i}^{2}}\|M_{i}(s,.)\|_{\L^{1}} ^{2}\,  \d s \right)^{\frac{1}{2}} = \frac{r_{i}}{|x-c_{i}|^{n+1}} \left( \int _{0} ^{\infty} r_{i}^{-2}\|M_{i}(s,.)\|_{\L^{1}} ^{2} \, \d s \right)^{\frac{1}{2}}.
\end{align*}
We now estimate $G_{i}$. We note that, for $y \not \in 2B_{i}$ and $z \in B_{i}$, we have that $|y-z| \geq \frac{|y-c_{i}|}{2}$. We thus have that
\begin{align*}
G_{i}(x) &:= \left( \int _{0} ^{r_{i}^{2}} \int _{s} ^{\infty} \fint _{B(x,\sqrt{t})}    
1_{(2B_{i})^{c}}(y)
\bigg(\int _{B_{i}} \frac{e^{-c\frac{|y-z|^{2}}{t-s}}}{(t-s)^{\frac{n}{2}+1}} |M_{i}(s,z)|\, \d z\bigg)^{2}
\d y\, \d t\, \d s \right)^{\frac{1}{2}} \\
 & \leq  \left( \int _{0} ^{r_{i}^{2}} \int _{s} ^{\infty} \fint _{B(x,\sqrt{t})}   
1_{(2B_{i})^{c}}(y)
\bigg(\int _{B_{i}} \frac{e^{-\frac{c}{4}\frac{|y-c_{i}|^{2}}{t-s}}}{(t-s)^{\frac{n}{2}+1}} |M_{i}(s,z)|\,  \d z\bigg)^{2}
\d y\, \d t\, \d s \right)^{\frac{1}{2}} \\
& = \left( \int _{0} ^{r_{i}^{2}} \int _{s} ^{\infty} \fint _{B(x,\sqrt{t})}   
\frac{e^{-\frac{c}{2}\frac{|y-c_{i}|^{2}}{t-s}}}{(t-s)^{n+2}} \|M_{i}(s,.)\|_{\L^{1}}^{2}\,
\d y\, \d t\, \d s \right)^{\frac{1}{2}}=\colon H_{i}(x).
\end{align*}
To estimate this integral, we split further the range of integration  according to  the following 4 subregions
\begin{align*}
R_{1}&:= \{(t,s,y) \in \R_{+}\times \R_{+} \times \R^{n} \;;\; t \geq 4r_{i}{^2}, \;  r_{i} \geq |x-y| \},\\
R_{2}&:= \{(t,s,y) \in \R_{+}\times \R_{+} \times \R^{n} \;;\; t \geq 4r_{i}{^2}, \;  r_{i} \leq |x-y| , \; |y-c_{i}| \leq |x-y|\}, \\
R_{3}&:= \{(t,s,y) \in \R_{+}\times \R_{+} \times \R^{n} \;;\; t \geq 4r_{i}{^2}, \;  r_{i} \leq |x-y|, \; |y-c_{i}| \geq |x-y|\}, \\
R_{4}&:= \{(t,s,y) \in \R_{+}\times \R_{+} \times \R^{n} \;;\; t \leq 4r_{i}{^2}\}.
\end{align*}
 In $R_{1}$, we have $t-s\ge \frac 3 4 t$ and $|y-c_{i}| \eqsim |x-c_{i}|$ as  $x\notin 4B_{i}$ and $y\in B(x,r_{i})$, hence
\begin{align*}
 &\left( \int _{0} ^{r_{i}^{2}} \int _{s} ^{\infty} \fint _{B(x,\sqrt{t})}   1_{R_{1}}(t,s,y)
\frac{e^{-\frac{c}{2}\frac{|y-c_{i}|^{2}}{t-s}}}{(t-s)^{n+2}} \|M_{i}(s,.)\|_{\L^{1}}^{2}\,
\d y\, \d t\,  \d s \right)^{\frac{1}{2}} \\
& \leq \left( \int _{0} ^{r_{i}^{2}} \bigg(\int _{\R^{n}} 1_{B(x,r_{i})}(y)  
\int _{4r_{i}^{2}} ^{\infty} \frac {
e^{-\frac{2c}{3}\frac{|y-c_{i}|^{2}}{t}}}{t^{3n/2+1}}\, \frac{\d t}{t}\, \d y\bigg) \|M_{i}(s,.)\|_{\L^{1}}^{2}\, \d s \right)^{\frac{1}{2}} \\
& \lesssim \left( \int _{0} ^{r_{i}^{2}} \bigg(\int _{B(x,r_{i})}   
|y-c_{i}|^{-(3n+2)}\, \d y\bigg) \|M_{i}(s,.)\|_{\L^{1}}^{2}\, \d s \right)^{\frac{1}{2}} \\
& \lesssim \frac{r_{i}^{n/2}}{ |x-c_{i}|^{3n/2+1}} \left( \int _{0} ^{r_{i}^{2}}\|M_{i}(s,.)\|_{\L^{1}}^{2}\, \d s \right)^{\frac{1}{2}} \leq \frac{r_{i}}{|x-c_{i}|^{n+1}} \left( \int _{0} ^{\infty} r_{i}^{-2}\|M_{i}(s,.)\|_{\L^{1}} ^{2}\, \d s \right)^{\frac{1}{2}}.
\end{align*}

In $R_{2}$, we use again $t-s\ge \frac 3 4 t$ and  obtain\begin{align*}
 &\left( \int _{0} ^{r_{i}^{2}} \int _{s} ^{\infty} \fint _{B(x,\sqrt{t})}   1_{R_{2}}(t,s,y)
\frac{e^{-\frac{c}{2}\frac{|y-c_{i}|^{2}}{t-s}}}{(t-s)^{n+2}} \|M_{i}(s,.)\|_{\L^{1}}^{2}\,
\d y\, \d t\, \d s \right)^{\frac{1}{2}} \\
 &\lesssim \left( \int _{0} ^{r_{i}^{2}} \int _{\R^{n}}   \int _{|x-y|^{2}} ^{\infty} t^{-n/2} 1_{R_{2}}(t,s,y)
\frac{e^{-\frac{2c}{3}\frac{|y-c_{i}|^{2}}{t}}}{t^{n+2}} \|M_{i}(s,.)\|_{\L^{1}}^{2}\,
\d t\, \d y\, \d s \right)^{\frac{1}{2}}\\
 &\lesssim \left( \int _{0} ^{r_{i}^{2}} \int _{\R^n}   \int _{|x-y|^{2}} ^{\infty} 1_{R_{2}}(t,s,y)
 \frac{e^{-\frac{2c}{3}\frac{|y-c_{i}|^{2}}{t}}}{t^{3n/2+2}} \|M_{i}(s,.)\|_{\L^{1}}^{2}\,
\d t\, \d y\, \d s \right)^{\frac{1}{2}} \\
 &\lesssim \left( \int _{0} ^{r_{i}^{2}} \bigg(\int _{\R^n} 1_{|y-x|\geq |y-c_{i}|}(y)
  |y-c_{i}|^{-(3n+2)} \bigg(\int _{0} ^{\frac{|y-c_{i}|^{2}}{|x-y|^{2}}} u^{3n/2+1}\frac{\d u}{u}\bigg) \d y\bigg)
 \|M_{i}(s,.)\|_{\L^{1}}^{2}\,
 \d s \right)^{\frac{1}{2}} \\
 &\lesssim \left( \int _{0} ^{r_{i}^{2}} \bigg( \int _{\R^n}1_{|y-x|\geq \frac{|x-c_{i}|}{2}}(y)|x-y|^{-(3n+2)}\d y \bigg)
 \|M_{i}(s,.)\|_{\L^{1}}^{2}\,
 \d s \right)^{\frac{1}{2}} \\
  &\lesssim \left( \int _{0} ^{r_{i}^{2}} |x-c_{i}|^{-(2n+2)} \|M_{i}(s,.)\|_{\L^{1}}^{2}\,
 \d s \right)^{\frac{1}{2}}\leq \frac{r_{i}}{|x-c_{i}|^{n+1}} \left( \int _{0} ^{\infty} r_{i}^{-2}\|M_{i}(s,.)\|_{\L^{1}} ^{2}\,  \d s \right)^{\frac{1}{2}}.
\end{align*}
The reasoning in $R_{3}$ is similar using $t-s\ge \frac 3 4 t$ and we have 
\begin{align*}
 &\left( \int _{0} ^{r_{i}^{2}} \int _{s} ^{\infty} \fint _{B(x,\sqrt{t})}   1_{R_{3}}(t,s,y)
\frac{e^{-\frac{c}{2}\frac{|y-c_{i}|^{2}}{t-s}}}{(t-s)^{n+2}} \|M_{i}(s,.)\|_{\L^{1}}^{2}\,
\d y \,\d t\, \d s \right)^{\frac{1}{2}} \\
 &\lesssim \left( \int _{0} ^{r_{i}^{2}} \int _{\R^{n}}   \int _{|x-y|^{2}} ^{\infty} t^{-n/2} 1_{R_{3}}(t,s,y)
\frac{e^{-\frac{2c}{3}\frac{|y-c_{i}|^{2}}{t}}}{t^{n+2}} \|M_{i}(s,.)\|_{\L^{1}}^{2}\,
\d t\, \d y\, \d s \right)^{\frac{1}{2}}\\
 &\lesssim \left( \int _{0} ^{r_{i}^{2}} \int _{\R^{n}}   \int _{0} ^{\infty} 1_{ 
 |y-x|\leq |y-c_{i}|}(y)
 \frac{e^{-\frac{2c}{3}\frac{|y-c_{i}|^{2}}{t}}}{t^{3n/2+2}} \|M_{i}(s,.)\|_{\L^{1}}^{2}\,
\d t\, \d y\, \d s \right)^{\frac{1}{2}} \\
 &\lesssim \left( \int _{0} ^{r_{i}^{2}} \int _{\R^n} 1_{ 
 |y-x|\leq |y-c_{i}|}(y) |y-c_{i}|^{-(3n+2)} 
 \|M_{i}(s,.)\|_{\L^{1}}^{2}\,
\d y\, \d s \right)^{\frac{1}{2}} \\
 &\lesssim \left( \int _{0} ^{r_{i}^{2}} \bigg( \int _{\R^n}1_{|y-c_{i}|\geq\frac{|x-c_{i}|}{2}}(y)|y-c_{i}|^{-(3n+2)}\d y \bigg)
 \|M_{i}(s,.)\|_{\L^{1}}^{2}\,
 \d s \right)^{\frac{1}{2}} \\
  &\lesssim \left( \int _{0} ^{r_{i}^{2}} |x-c_{i}|^{-(2n+2)} \|M_{i}(s,.)\|_{\L^{1}}^{2}\,
 \d s \right)^{\frac{1}{2}}\leq \frac{r_{i}}{|x-c_{i}|^{n+1}} \left( \int _{0} ^{\infty} r_{i}^{-2}\|M_{i}(s,.)\|_{\L^{1}} ^{2}\,  \d s \right)^{\frac{1}{2}}.
\end{align*}
To complete the proof of \eqref{I(x)est1}, it just remains to estimate $H_{i}$ in the subregion $R_{4}$. This time we use the trivial estimate $t\ge t-s$, and   $|y-c_{i}| \eqsim |x-{c_{i}}|$ when $x\notin 4B_{i}$ and $|y-x|^2 \le  t \le 4r_{i}^2$ to obtain
\begin{align*}
&\left( \int _{0} ^{r_{i}^{2}} \int _{s} ^{\infty} \fint _{B(x,\sqrt{t})}   1_{R_{4}}(t,s,y)
\frac{e^{-\frac{c}{2}\frac{|y-c_{i}|^{2}}{t-s}}}{(t-s)^{n+2}} \|M_{i}(s,.)\|_{\L^{1}}^{2}\,
\d y\, \d t \, \d s \right)^{\frac{1}{2}} \\
& =  \left( \int _{0} ^{r_{i}^{2}} \int _{s} ^{4r_{i}^{2}} \fint _{B(x,\sqrt{t})}   1_{R_{4}}(t,s,y)
\frac{e^{-\frac{c}{2}\frac{|y-c_{i}|^{2}}{t-s}}}{(t-s)^{n+2}} \|M_{i}(s,.)\|_{\L^{1}}^{2}\,
\d y \,\d t\, \d s \right)^{\frac{1}{2}} \\
& \lesssim \left( \int _{0} ^{r_{i}^{2}} \int_{\R^n}  \int _{\max(s,|x-y|^{2})} ^{4r_{i}^{2}} (t-s)^{-n/2} 1_{|x-y|^{2}\leq 4r_{i}^{2}}(y)
\frac{e^{-\frac{c}{2}\frac{|y-c_{i}|^{2}}{t-s}}}{(t-s)^{n+2}} \|M_{i}(s,.)\|_{\L^{1}}^{2}\,
\d t \,\d y \,\d s \right)^{\frac{1}{2}} \\
& \lesssim \left( \int _{0} ^{r_{i}^{2}} \bigg( \int_{B(x,2r_{i})}  
\bigg( \int _{0} ^{\infty}  \frac{e^{-\frac{c}{2}\frac{|y-c_{i}|^{2}}{t}}}{t^{3n/2+1}} \frac{\d t}{t} \bigg) \d y \bigg)\|M_{i}(s,.)\|_{\L^{1}}^{2}\,
 \d s \right)^{\frac{1}{2}} \\
 & \lesssim \left( \int _{0} ^{r_{i}^{2}} \bigg( \int_{B(x,2r_{i})}  
|y-c_{i}|^{-(3n+2)} \d y \bigg)\|M_{i}(s,.)\|_{\L^{1}}^{2}\,
 \d s \right)^{\frac{1}{2}} \\
  &\lesssim \frac{r_{i}^{n/2}}{|x-c_{i}|^{3n/2+1}} \left( \int _{0} ^{r_{i}^{2}}  \|M_{i}(s,.)\|_{\L^{1}}^{2}\,
 \d s \right)^{\frac{1}{2}}\leq \frac{r_{i}}{|x-c_{i}|^{n+1}} \left( \int _{0} ^{\infty} r_{i}^{-2}\|M_{i}(s,.)\|_{\L^{1}}^{2}\, \d s \right)^{\frac{1}{2}}.
\end{align*}
Now that \eqref{I(x)est1} is proven, we prove the final estimate \eqref{I(x)est2}.
For $i \in \mathbb{N}$ and $s>0$, we first apply Poincar\'e's inequality and the construction of $M_{i}$ to note that
\begin{align*}
r_{i}^{-1}\fint _{B_{i}} |M_{i}(s,y)|\,  \d y 
&\lesssim \fint _{B_{i}} |\nabla M_{i}(s,y)|\, \d y \leq  \bigg(\fint _{B_{i}} |\nabla M_{i}(s,y)|^{p} \, \d y\bigg)^{\frac{1}{p}} \\
&\lesssim \bigg(\fint _{B_{i}} |\widetilde{b}_{i}(s,y)|^{p} \,\d y\bigg)^{\frac{1}{p}} \lesssim \bigg(\fint _{B_{i}} |b_{i}(s,y)|^{p}\, \d y\bigg)^{\frac{1}{p}}.
\end{align*}
Using this fact, we obtain
\begin{align*}
&\left(\int _{0} ^{r_{i}^{2}} r_{i}^{-2}\|M_{i}(s,.)\|_{\L^{1}}^{2}\, \d s\right)^{\frac{1}{2}} \lesssim
|B_{i}| \left(\int _{0} ^{r_{i}^{2}}\bigg(r_{i}^{-1} \fint _{B_{i}} |M_{i}(s,y)| \d y\bigg)^{2} \d s\right)^{\frac{1}{2}} 
\\
&\quad \lesssim |B_{i}| \left(\int _{0} ^{r_{i}^{2}} \bigg(\fint _{B_{i}} |b_{i}(s,y)|^{p}\, \d y\bigg )^{\frac{2}{p}}\, \d s\right)^{\frac{1}{2}} 
\\
&
\quad
\lesssim |B_{i}| \left(\int _{0} ^{r_{i}^{2}} \bigg(\fint _{2B_{i}} \fint _{B(z,\sqrt{s})} |b_{i}(s,y)|^{p}\, \d y\,\d z\bigg )^{\frac{2}{p}} \d s\right)^{\frac{1}{2}} 
\\
& \quad = |B_{i}| \left[ \left(\int _{0} ^{r_{i}^{2}} \bigg(\fint _{2B_{i}} \fint _{B(z,\sqrt{s})} |b_{i}(s,y)|^{p}\, \d y\, \d z\bigg )^{\frac{2}{p}} \d s \right)^{\frac{p}{2}}\, \right]^{\frac{1}{p}} \\
& \quad \leq |B_{i}| \left[ \fint _{2B_{i}} \left(\int _{0} ^{r_{i}^{2}} \bigg( \fint _{B(z,\sqrt{s})} |b_{i}(s,y)|^{p}\, \d y\bigg )^{\frac{2}{p}}\, \d s \right)^{\frac{p}{2}}\, \d z \right]^{\frac{1}{p}} 
\\
&
\quad
\lesssim |B_{i}| \left[ \fint _{2B_{i}} \mathcal{A}(b_{i})(z)^{p}\, \d z \right]^{\frac{1}{p}} \\
& \quad \lesssim |B_{i}|^{1-1/p} \|b_{i}\|_{\T^{p}} \lesssim \lambda |B_{i}|,
\end{align*}
using \eqref{badnorm} in the last inequality to conclude the proof of \eqref{I(x)est2}.
\end{proof}

\section{Existence, uniqueness, and regularity of pathwise weak solutions}
\label{sec:stoch}
In order to obtain our well-posedness result, we set  Problem  \eqref{eq:SPDE} mentioned in the Introduction, which we recall here for convenience. Consider 
\begin{equation*}
\left\{
  \begin{array}{ll}
   \d U(t) - \div a(t,.) \nabla U(t) \, \d t & = \div F(t) \, \d t +  g(t)\, \d W_{\H}(t),  \quad t\in \R_{+}\\
    U(0) & = \psi,
  \end{array}
\right.
\end{equation*}
driven by a cylindrical Brownian motion $W_{\H}$ on a complex separable Hilbert space $\H$, with $a \in \L^{\infty}(\Omega \times\R_{+}; \L^{\infty}(\R^{n}; \mathcal{L}(\R^{n})))$ (non necessarily symmetric) satisfying the uniform ellipticity condition:  for some $C>0$, 
\begin{equation}
\label{eq:uniformellipticity}
 a(\omega, t,x) \xi \cdot\xi  \geq C |\xi|^{2} \quad \forall \xi \in \R^{n},
\end{equation}
for almost every $(\omega, t,x) \in \Omega \times \R_{+} \times \R^{n}$. The space of randomness $\Omega$ on which we realise the Brownian motion is equipped with a probability measure, which is implicit in the sequel. On the other hand, the measures on $\R_{+}$ and $\R^n$ are Lebesgue measures.

  We denote by $(\mathcal{F}_{t})_{t \geq 0}$ the filtration associated with our cylindrical Brownian motion, and require $g$ to be (locally in space-time) a limit of simple  processes adapted to this filtration (defined next).  We do not require $a$ or $F$ to be adapted.

\begin{defn}
\label{def:sap}
We call  $g:  \Omega  \times \R_{+} \times \R^{n}  \to \H$ a {\em simple adapted process} if it is of the form 
\begin{equation}
\label{eq:sap}
g(\omega, t,x) = \sum _{\ell= 0  } ^{N}   1_{(t_{\ell},t_{\ell+1}]}(t)
\sum _{m=1} ^{M} \one_{A_{m,\ell}}(\omega) \sum _{k=1} ^{K} \phi_{k,m,\ell}(x)\,   h_{k}, 
\end{equation}
with $N,M,K \in \mathbb{N}$, $0=t_{0}\leq t_{1}<...<t_{N+1}$, $A_{m,\ell} \in \mathcal{F}_{t_{\ell}}$, $\phi_{k,m,\ell}  \in \C_{c} ^{\infty}(\R^{n})$,  and $h_{k} \in \H$. \end{defn}

For a simple adapted process  and  $t>0$,  the stochastic integral
$\int _{0} ^{t} g(s,.)\, \d W_{\H}(s)$   is then understood as in the standard theory for $\L^{2}$-valued processes (see \textit{e.g.,} \cite{rl} or \cite{nvw}). More precisely, for $g$ as in \eqref{eq:sap}, $t>0$  and $x \in \R^{n}$, 
$$
\bigg(\int _{0} ^{ t  } g(s,.)\, \d W_{\H}(s)\bigg)(x) := 
\sum^{ N}_{\ell= 0 }\sum _{m=1} ^{M} \one_{A_{m,\ell}} \sum _{k=1} ^{K} \phi_{k,m,\ell}(x)  \big(W_{\H}( t\wedge  t_{\ell+1})\,h_{k} 
- W_{\H}( t\wedge   t_{\ell})\,h_{k}\big).
$$

\subsection{Pathwise weak solutions}

\begin{defn}
\label {def:pathweak}
Let $\psi\in \L^1(\Omega; \L^1_{\loc}(\R^n))$ and $F\in \L^1(\Omega; \L^1_{\loc}([0,+\infty); \L^{1}_{\loc}(\R^n)))$. Let  $ g \in \L^1(\Omega; \L^1_{\loc}([0,+\infty); \L^{1}_{\loc}(\R^n;\H))$ 
 belong,  for some $q \in (1,2]$ and all $R,T>0$, 
 to the $\L^{q}(\Omega \times B(0,R);\L^{2}([0,T];\H))$ closure of simple adapted processes.
We say that 
$U\in\L^1(\Omega; \L^1_{\loc}([0,+\infty); \W^{1,1}_{\loc}(\R^n)))$
is a {\em pathwise weak solution} of Problem \eqref{eq:SPDE} if the following property $(\PW)$    holds:

 $(\PW)\colon $  For all (non random) test functions  $\varphi\in \C^{\infty}_{c}([0,\infty)\times \R^{n})$, 
the random functions
\begin{align*}
& V_{\varphi} \colon t\mapsto \langle U(t,.),\varphi(t,.)\rangle - \bigg\langle \int  _{0} ^{t} g(s,.)\, \d W_{\H}(s), \varphi(t,.) \bigg\rangle  
\end{align*}
and
\begin{align*}
& \mathbb{V}_{\varphi} \colon t\mapsto \langle \psi, \varphi(0,.) \rangle 
-
\int  _{0} ^{t}  \langle a(s,.)\nabla U(s,.),\nabla \varphi(s,.) \rangle \, \d s +  \int  _{0} ^{t} \langle
U(s,.),\partial_{s}\varphi(s,.) \rangle\,  \d s \\ &
\qquad  \quad -  
\int  _{0} ^{t} \langle
F(s,.),\nabla \varphi(s,.) \rangle\,  \d s  
- \int  _{0} ^{t} \bigg\langle \int  _{0} ^{s} g(\tau,.)\, \d W_{\H}(\tau), \partial_{s}\varphi(s,.) \bigg\rangle\,  \d s,
\end{align*} 
 agree in $\L^{1}(\Omega\times [0,T])$ for all $T>0$. 
 \end{defn}
 
The bracket terms   $\langle f,\phi \rangle:= \int _{\R^{n}} f(x)\phi(x)\, \d x$  are integrals in the sense of Lebesgue. 

The following lemma explains the definition of all terms, and shows that,  
although a pathwise weak solution is merely a random function rather than a process, 
it can be interpreted  in terms of 
two families of continuous processes $V_{\varphi}$ and $\mathbb{V}_{\varphi}$, indexed by the test functions $\varphi \in \C^{\infty}_{c}([0,\infty)\times \R^{n})$.

\begin{lem} 
 \label{lem:pathweak}
 Assume the conditions of Definition \ref{def:pathweak}.  Fix $\varphi\in \C^{\infty}_{c}([0,\infty)\times \R^{n})$.
 \begin{enumerate}
 \item For all $T,R>0$,  $t\mapsto \int  _{0} ^{t} \,    g(s,.)\,  \d W_{\H}(s)$  is a  continuous  adapted process in $\L^{q}(\Omega; \C([0,T]  ;\L^{q}(B(0,R))))$.
  \item   All bracket terms in  $V_{\varphi}$ are Lebesgue integrals  and  $\V_{\varphi} \in \L^{1}(\Omega\times [0,T])$  for all $T>0$.
  \item All bracket terms in  $\mathbb{V}_{\varphi}$  are Lebesgue integrals  and $\mathbb{V}_{\varphi} \in  \W^{1,1}([0,T]; \L^1(\Omega))$ for all $T>0$. 
  \item If $U$ is a pathwise weak solution of Problem \eqref{eq:SPDE}, the $\L^{1}(\Omega)$-valued map $V_{\varphi}$ has an absolutely continuous representative in $[0,\infty)$, equal to $\mathbb{V}_{\varphi}$.  In particular, $\lim_{t\to 0}V_{\varphi}(t,.)=\langle \psi, \varphi(0,.) \rangle$  almost surely.  \end{enumerate}

\end{lem}

\begin{proof} We begin with (i). 
 Let $T,R>0$ and fix $1< q \le 2$. For a simple adapted process $g$, the classical (scalar-valued) Burkh\"older-Davis-Gundy inequality  (see, \textit{\textit{e.g.,}} \cite[Theorem 3.28]{ks})  gives, for almost all $x$,
 \begin{align*}
 \mathbb{E} \underset{t \in [0,T]}{\sup} \bigg| \int  _{0} ^{t}   g(s,x)\,  \d W_{\H}(s)\bigg|^{q}  
 \lesssim 
  \mathbb{E} \| g(.,x) \|_{\L^{2}([0,T];\H)}^{q}.
\end{align*}
Integrating on $B(0,R)$ and using Fubini theorem give the following estimates:
\begin{align*}
&\mathbb{E}  \underset{t \in [0,T]}{\sup} 
\bigg\|\int  _{0} ^{t}  
 g(s,.)\, \d W_{\H}(s)\bigg\|_{\L^{q}(B(0,R))}^{q} 
\leq
 \int _{B(0,R)} \mathbb{E} \underset{t \in [0,T]}{\sup} 
 \bigg| \int  _{0} ^{t}  g(s,x)\,  \d W_{\H}(s)\bigg|^{q} \, \d x \\
&\quad   \lesssim 
 \int _{B(0,R)} \mathbb{E} \| g(.,x) \|_{\L^{2}([0,T];\H)}^{q}\, \d x
 = \mathbb{E} \|g \|_{\L^{q}(B(0,R);\L^{2}([0,T];\H))}^{q}.
\end{align*}
 Now consider a general 
$ g \in \L^1(\Omega; \L^1_{\loc}([0,+\infty); \L^{1}_{\loc}(\R^n;\H))$ that
belongs, for all $T,R>0$,
 to the $\L^{q}(\Omega \times B(0,R);\L^{2}([0,T];\H))$ closure of simple adapted processes.
One can find a sequence of  
simple adapted processes $(g_{T,R,m})_{m \in \mathbb{N}}$ converging to $g$ in $\L^{q}(\Omega \times B(0,R);\L^{2}([0,T];\H))$ norm. A priori, this sequence depends on $T,R$ but doing this for each $T=R=k\ge1$ and using the diagonal principle, one can extract a  sequence independent of $T,R>0$ (observe that the spaces get smaller as each of $T,R$ increases).   Let us call it $(g_{m})_{m \in \mathbb{N}}$.
 By the above estimates, we have that 
 $$ \lim_{m,m'\to \infty}  \mathbb{E}\bigg\|t \mapsto \int  _{0} ^{t}  \big(g_{m'}(s,.)-g_{m}(s,.)\big)\,  \d W_{\H}(s)\bigg\|^q_{ \C([0, T] ;\L^{q}(B(0,R))} = 0. $$
This gives us that $t \mapsto  \int  _{0} ^{t}  g(s,.)\,  \d W_{\H}(s)$ defines 
a continuous adapted $\L^q(B(0,R))$-valued process.

We turn to (ii).  Let $T,R>0$  be  such that $\varphi$ is supported in $[0,T]\times B(0,R)$. Being a $\L^q(B(0,R))$-valued process on $[0,T]$, 
$(\omega, t,x)\mapsto 
 \int  _{0} ^{t}   g(s,.) \,  \d W_{\H}(s)
 (\omega, x)$ 
 is measurable  on $\Omega \times [0,T] \times B(0,R)$,  and is Lebesgue integrable in $x$ for almost all $(\omega, t)$. 
Given the assumptions on $U$ and $\varphi$, 
both  bracket  terms appearing in $V_{\varphi}$ are thus well-defined as Lebesgue integrals in $x$, and are measurable as functions of $(\omega, t)$.   
With the help of  H\"older's inequality and  the estimates proven in (i),  we obtain the following estimates: 
\begin{align}
\label{eq:defdet}
\mathbb{E} \int _{0} ^{T} |\langle U(t,.),\varphi(t,.)\rangle| \, \d t & \lesssim \|U\|_{\L^{1}(\Omega\times [0,T]\times B(0,R))} \|\varphi\|_{\L^{\infty}}, \\
\label{eq:defstoch}
 \mathbb{E} \int _{0} ^{T} \bigg|\bigg\langle \int  _{0} ^{t} g(s,.)\, \d W_{\H}(s), \varphi(t,.) \bigg\rangle \bigg| \, \d t & \lesssim   \|g\|_{\L^{q}(\Omega \times B(0,R);\L^{2}([0,T];\H))}   \|\varphi\|_{\L^{1}([0,T];\L^{q'}(\R^{n}))}.
\end{align}
 
Turning to (iii),  and  considering $\varphi$ and $T,R$ related as in (ii), the case of the bracket with $\psi$ is trivial and we consider the other four terms.  The above reasoning, together with the conditions on $F$, also gives that all brackets are well-defined as Lebesgue integrals in $x$, and are measurable as functions of $(\omega, s)$. Because we further integrate in $s$ on $[0,t]$, it remains to prove that they belong to $\L^{1}( \Omega \times [0,T])$. We have 
\begin{align*}
\mathbb{E} \int _{0} ^{T} |\langle a(s,.)\nabla U(s,.),\nabla \varphi(s,.)\rangle| \, \d s & \lesssim \|\nabla U\|_{\L^{1}(\Omega\times [0,T]\times B(0,R))} \|\nabla \varphi\|_{\L^{\infty}},
\end{align*}
and  the other terms are treated as in \eqref{eq:defdet} and \eqref{eq:defstoch} with the obvious changes.

Concerning (iv), assuming $(\PW)$,  we have that $V_{\varphi}$ and $\mathbb{V}_{\varphi}$ agree as elements of $L^{1}(\Omega\times [0,T])$  for each $T>0$. By (iii), this means that, as an $L^{1}(\Omega)$-valued map, $V_{\varphi}$ has a (unique) absolutely continuous representative equal to $\mathbb{V}_{\varphi}$. Since $\lim_{t\to 0}\mathbb{V}_{\varphi}(t,.)=\langle \psi, \varphi(0,.) \rangle$ almost surely, the same is true for $V_{\varphi}$.
\end{proof}

\subsection{Stochastic integrals in tent spaces.}

For our well-posedness and maximal regularity result, we take $p \in (1,\infty)$, 
and consider data $(\psi, F,g) \in \L^{p}(\Omega; \mathcal{T}^{p})$, where
$$\L^{p}(\Omega; \mathcal{T}^{p}):=\L^{p}(\Omega; \L^{p})\times \L^{p}(\Omega; \T^{p}(\IC^{n})) \times \L^{p}(\Omega; \T^{p}(\H)),$$
and solutions in the space $\L^{p}(\Omega;\dot{\mathcal{T}}^{p}_{1})$ where
$$
\dot{\mathcal{T}}^{p}_{1}:= \{u \in \L^1_{\loc}([0,\infty)\times \R^{n})  \;;\; \nabla u \in \T^{p}(\IC^{n})\}.$$
For $p=2$, the data belong to the usual $\L^2$ space and the solutions have gradients in  $\L^2$. 
\begin{rem}
\label{rem:l2loc}
Note that $\nabla u \in \T^{p}(\IC^{n})$ with $u\in \L^{1}_{\loc}(\R_{+}\times \R^{n})$, and the help of Poincar\'e inequality imply $u
 \in \L^{2}_{\loc} (\R_{+}; \H^1_{\loc}(\R^{n}))$.  
\end{rem}
For some calculations it will be useful to use both $\T^p$ and $\V^p$ depending on the value of $p$. Define
\begin{align*}
\E^{p} : = \begin{cases} \V^{p} \quad \text{if} \; p \geq 2,\\
\T^{p} \quad \text{if} \; p<2,
\end{cases}
\qquad
\F^{p} : = \begin{cases} 
\T^{p} \quad \text{if} \; p\geq2,\\
\V^{p} \quad \text{if} \; p < 2,
\end{cases}
\end{align*}
for $1<p<\infty$.
In this range of $p$,  
 the spaces $\T^{p}, \V^p$ are reflexive (and UMD) with duals $\T^{p'}, \V^{p'}$ in the $\L^2(\R_{+}\times \R^n)$ duality, respectively, so the same is true for $\E^p$ and $\F^{p'}$. Moreover, as noted in the   introduction  $\E^{p} \hookrightarrow \T^{p} \hookrightarrow \F^{p}$.  \\

\begin{defn}
\label{def:adapted}
 Let $1<p<\infty$. We say that a random $\T^{p}(\H)$-valued 
function is {\em adapted} if it belongs to the closure, in the $\L^{p}(\Omega; \T^{p}(\H))$ norm, of the set of simple adapted processes. 
\end{defn} 

This definition expresses the new paradigm of working with tent spaces: their norms involve local integrals in both time and space variables simultaneously. 
Of course, simple adapted processes    belong to $\L^{p}(\Omega; \T^{p}(\H))$. Indeed, bounded functions with compact support in $[0,\infty)\times \R^n$ belong to $\T^p$ as one can easily check, and one concludes by linearity from \eqref{eq:sap}.
Note that  $\T^p$, $p\ne 2$, is not of the form $\L^q(\I;\B)$, where $\B$ would be a Banach space $\B$ (like $\L^p$) and $\I$ an interval, so the completion is not exactly the same as in the definition of UMD-valued adapted processes as limits of elementary processes in \cite[Section 2.4]{nvw}.  
 The  content of the following lemma is to extend the meaning of   the It\^o   integral to   all adapted $g \in \L^{p}(\Omega;\T^{p}(\H))$.

\begin{lem} 
\label{lem:Smap}
Let $1<p<\infty$, $T>0$, and $g:  \Omega \times \R_{+} \times \R^{n}  \to \H$ be a simple adapted process. Then  $$
\mathbb{E}\bigg\|(t,x) \mapsto  \one_{[0,T]}(t)\bigg( \int _{0} ^{t} g(s,.)\, \d W_{\H}(s)\bigg)(x)\bigg\|^{p} _{\F^{p}}
\lesssim  T^{p/2} \mathbb{E} \|g\|^{p} _{\T^{p}(\H)}.$$
Consequently,$$ (t,x) \mapsto \one_{[0,T]}(t)\bigg(\int _{0} ^{t} g(s,.)\, \d W_{\H}(s)\bigg)(x)$$ 
extends by density for all adapted  $g \in \L^{p}(\Omega;\T^{p}(\H))$ to an element in $\L^p(\Omega;\F^p)$. 
\end{lem}
\begin{proof}
 Let $g$ be a simple adapted process with the notation as in \eqref{eq:sap}. 
 For all $s>0$, define $$\tilde{g}(s)\colon (t,x)\mapsto  \one_{[0,T]}(t)  \, \one_{[0,t]}(s)\, g(s,x).$$
Then $s\mapsto \tilde{g}(s)$  is an adapted  $\L^{2}(\R_{+}\times \R^{n}; \H)$-valued process. Indeed, for $s> t_{N+1}$, $\tilde g(s)=0$ and for all 
$\ell=0,\ldots,N$ and $s \in (t_{\ell},t_{\ell+1}]$,
$$
\tilde{g}(s)  (\omega, t ,x)  =  \one_{[0,T]}(t)  \,   \one_{[0,t]}(s) 
\sum _{m=1} ^{M} \sum _{k=1} ^{K} \one_{A_{m,\ell,k}}(\omega)\, \phi_{m,\ell,k}(x)\, h_{k},
$$ 
so that one sees that it is   $\mathcal{F}_{t_{\ell}} \subset \mathcal{F}_{s}$ measurable  as an $\L^{2}(\R_{+}\times \R^{n}; \H)$-valued random variable.  
Consequently, $(t,x)\mapsto \big( \int  _{0} ^{\infty} \tilde{g}(s) \, \d W_{\H}(s)\big) (t,x)$
defines a random element of $\L^{2}(\R_{+}\times \R^{n})= \T^{2}$ using the It\^o integral of $\L^{2}(\R_{+}\times \R^{n}; \H)$-valued processes.   
 Using the It\^o isomorphism for stochastic integrals from
 \cite[Theorem 3.5]{nvw}  and the square function estimate 
\cite[Proposition 2.1]{nvw1} in the UMD space $\F^p$ applied to $\tilde g$,  we have that 
\begin{align*}
&\mathbb{E}\bigg\|(t,x) \mapsto   \bigg(\int _{0} ^{\infty}  \tilde{g}(s)\,  \d W_{\H}(s)\bigg)(t,x)\bigg\|^{p} _{\F^{p}}
\\
& \qquad \lesssim \mathbb{E}\bigg\|(t,x) \mapsto \bigg(\int _{0} ^{\infty} | \tilde{g}(s) (t,x)|_{\H}^{2}\,  \d s\bigg)^{\frac{1}{2}}\bigg\|^{p}_{\F^{p}}.
\end{align*}
Reexpressing in terms of $g$ using
 $$
\one_{[0,T]}(t)  \bigg(\int  _{0} ^{t} {g}(s,.)\, \d W_{\H}(s)\bigg) (x)=\bigg(\int_{0}^{\infty}  \tilde{g}(s)\,  \d W_{\H}(s)\bigg) (t, x),$$ 
this gives us 
\begin{align*}
&\mathbb{E}\bigg\|(t,x) \mapsto  \one_{[0,T]}(t) \bigg(\int _{0} ^{t} g(s,.)\, \d W_{\H}(s)\bigg)(x)\bigg\|^{p} _{\F^{p}}
\\
& \qquad \lesssim \mathbb{E}\bigg\|(t,x) \mapsto \one_{[0,T]}(t)\bigg(\int _{0} ^{t} |g(s,x)|_{\H}^{2}\,  \d s\bigg)^{\frac{1}{2}}\bigg\|^{p}_{\F^{p}}  =  \mathbb{E} \|S(|g|_{\H})\|^{p} _{\F^{p}}, 
\end{align*}
where $S$ is the deterministic sublinear operator defined by
$$
S(h):(t,x)\mapsto 1_{[0,T]}(t) \bigg(\int _{0} ^{t} |h(s,x)|^{2}\, \d s\bigg)^{{1/2}}. 
$$
For $1<p<\infty$, this operator is bounded on $\V^{p}$ since 
\begin{align*}
\int _{\R^{n}} \bigg(\int _{0} ^{T} \int _{0} ^{t} |h(s,x)| ^{2}\, \d s\, \d t\bigg)^{{p/2}}\, \d x
\leq T^{{p/2}} \int _{\R^{n}} \bigg(\int _{0} ^{\infty} |h(s,x)| ^{2} \, \d s \bigg)^{{p/2}}\, \d x.
\end{align*}
This implies that $S$ is bounded on $\F^{p}$ for $p \leq 2$.
To prove that $S$ is also bounded on $\F^{p}$ for $p \geq 2$, we just need to show that it is bounded on $\T^{\infty}$, by real interpolation.
This holds because, given any ball $B$ of radius $r_{B}$, and any function $h \in \T^{\infty}$, we have 
\begin{align*}
 \fint _{B} \int _{0} ^{r_{B}^2 \wedge T} \int _{0} ^{t} |h(s,y)|^{2}\, \d s\, \d t \,\d y 
\leq T  \fint _{B} \int _{0} ^{r_{B}^2}  |h(s,y)|^{2}\, \d s\,  \d y \leq T^{} \|h\|_{\T^{\infty}}^2.
\end{align*}
Therefore $\|S(h)\|_{\T^{\infty}} \le T^{{1/2}} \|h\|_{\T^{\infty}}$ as desired. In conclusion, $\|S(h)\|_{\F^{p}} \lesssim T^{{1/2}}\|h\|_{\F^{p}}$ for all $h \in \F^{p}$ and all $1<p<\infty$.

Applying this  estimate  extended to $\H$-valued random functions, we thus have 
\begin{align*} 
 \mathbb{E} \|S(|g|_{\H})\|^{p} _{\F^{p}}  
\lesssim T^{p/2} \mathbb{E}\||g|_{\H}\|^{p} _{\F^{p}} \lesssim T^{p/2}\mathbb{E} \|g\| ^{p} _{\T^{p}(\H)},
\end{align*}
where we have used $\T^{p} \hookrightarrow \F^{p}$, and this concludes the argument. 
\end{proof}

\begin{rem} It is likely not true that $S$ is bounded on $\T^p$ when $p<2$, and boundedness  from $\T^p$ to $\F^p$  is the  next best result. 
\end{rem} 

  Our concept of pathwise weak solution can be used for Problem \eqref{eq:SPDE} with data
 $(\psi,F,g) \in \L^{p}(\Omega; \mathcal{T}^{p})$ and g adapted. This is the content of the following Lemma.

 \begin{lem}
 \label{lem:pwsTp}
 Let $1<p<\infty$, $(\psi,F,g) \in \L^{p}(\Omega; \mathcal{T}^{p})$ with g adapted, and let $U \in \L^{p}(\dot{\mathcal{T}}^{p}_{1})$.  Then $\psi\in \L^1(\Omega; \L^1_{\loc}(\R^n ))$, $F\in \L^1(\Omega; \L^1_{\loc}([0,+\infty); \L^{1}_{\loc}(\R^n; \IC^n )))$, and $ g \in \L^1(\Omega; \L^1_{\loc}([0,+\infty); \L^{1}_{\loc}(\R^n);\H))$ is such that, for $q= \min(p,2)$, and all $T,R>0$,  $g$ 
belongs to the $\L^{q}(\Omega \times B(0,R);\L^{2}([0,T];\H))$ closure of simple adapted processes. 
Moreover, 
  $U\in\L^1(\Omega; \L^1_{\loc}([0,+\infty); \W^{1,1}_{\loc}(\R^n)))$.
\end{lem}

\begin{proof}
Since $\L^p(\Omega; \L^p) \subset \L^1(\Omega; \L^1_{\loc}(\R^n))$ by Jensen's inequality in $(\omega, x)$, we have that $\psi\in \L^1(\Omega; \L^1_{\loc}(\R^n))$. 

Next, we show that for all $T,R>0$, 
\begin{equation}
\label{eq:L1loc}
\T^p\subset  \L^{q}( B(0,R);\L^{2}([0,T])) \subset  \L^{1}( [0,T]\times B(0,R)),
\end{equation}
where $q=\min(p,2)$, with estimates. 
The second inclusion is a mere consequence of H\"older's inequality. We turn to the first one. Let $h$ be a  measurable function. 
For $p\ge 2$, this is based on the following averaging technique and Jensen inequality. 
\begin{align*}
 \int_{0}^T\int_{B(0,R)} &|h(t,y)|^2\, \d t\,\d y  \leq  \int_{B(0,R+\sqrt T)}\int_{0}^T \fint_{B(x, \sqrt t)} |h(t,y)|^2 \, \d t\,\d y \, \d x
 \\
 &\leq C(p,T,R) \bigg(\int_{B(0,R+\sqrt T)}\bigg(\int_{0}^T \fint_{B(x, \sqrt t)} |h(t,y)|^2 \, \d t\,\d y \bigg)^{p/2}\, \d x\bigg)^{2/p}. 
\end{align*}
For $p\le 2$, 
\begin{align*}
\int_{B(0,R)} \bigg( \int_{0}^T |h(t,y)|^2\, \d t\bigg)^{p/2}\,\d y \le \|h\|_{\V^p}^p \lesssim \|h\|_{\T^p}^p,
\end{align*}
since $ \T^{p} \hookrightarrow \V^{p}$ as noted in the Introduction.
If, now, $h$ depends also on $\omega$, then this yields
$$
\|h\|_{\L^{1}(\Omega \times [0,T] \times B(0,R))} \lesssim \|  h\|_{\L^{q}(\Omega \times B(0,R);\L^{2}([0,T]))}   \lesssim
\|h\|_{\L^{p}(\Omega;\T^{p})}. 
$$
Thus,  $F \in \L^p(\Omega; \T^p(\IC^n))$ yields in particular  $F\in \L^1(\Omega; \L^1_{\loc}([0,+\infty); \L^{1}_{\loc}(\R^n; \IC^n )))$. Similarly, this applies to $\nabla U$, hence the conclusion for $U$. 
Lastly, given 
an adapted process $g \in \L^{p}(\Omega;\T^{p}(\H))$, there is by definition a sequence $(g_{m})_{m \in \mathbb{N}}$ of simple adapted processes converging to $g$ in $ \L^{p}(\Omega;\T^{p}(\H))$. 
Then, the above estimate shows that for any $T,R>0$, the restriction of $g_{m} $ to $\Omega \times [0,T] \times B(0,R)$ converges to $g$ in $\L^{q}(\Omega \times B(0,R);\L^{2}([0,T];\H))$, hence that
$g$ belongs to the $\L^{q}(\Omega \times B(0,R);\L^{2}([0,T];\H))$ closure of simple adapted processes. 
\Bk 
%
%
\end{proof}

The above lemma thus shows that, for  $(\psi,F,g) \in \L^{p}(\Omega; \mathcal{T}^{p})$ and g adapted, $U \in \L^{p}(\Omega;\dot{\mathcal{T}}^{p}_{1})$ is a pathwise weak solution of Problem \eqref{eq:SPDE}  if and only if it satisfies property $(\PW)$ in Definition \ref{def:pathweak}.\\

The next result is for approximation purposes. 
\begin{lem}\label{lem:approxdata}
Let $(\psi,F,g) \in \L^{p}(\Omega; \mathcal{T}^{p})$, with $g$ adapted.  
 Then there exists a sequence $(\psi_{\ell},F_{\ell},g_{\ell}) \in \L^{p}(\Omega; \mathcal{T}^{p})\cap \L^{2}(\Omega;\mathcal{T}^{2})$, $\ell\ge1$, with $g_{\ell}$ being a simple adapted process, converging to $(\psi,F,g)$ in $ \L^{p}(\Omega; \mathcal{T}^{p})$ norm.
\end{lem}

\begin{proof}
For $\psi \in \L^p(\Omega; \L^p)$ this is standard procedure by truncations in $(\omega,x)$. 
 As $g$ is adapted in $\L^p(\Omega, \T^p(\H))$, we may  take  all $g_{\ell}$ to be simple adapted processes by definition, and they also belong to $\L^2(\Omega; \T^2(\H))$.  
For $F\in \L^p(\Omega;\T^p)$, first it is easy to see that  $F(\omega,.)$ multiplied by indicators $\one_{K_{\ell}}$, where $(K_{\ell})_{\ell\ge 0}$ is a non-decreasing sequence  of compact sets  exhausting $\R_{+}\times \R^n$, converges in $\T^p$ to $F(\omega,.)$. As $\|\one_{K_{\ell}}F(\omega,.)\|_{\T^p}\le \|F(\omega,.)\|_{\T^p}$, we can use  dominated convergence to conclude for convergence in $\L^{p}(\Omega; \T^{p})$. Second, we may assume that $F(\omega,.)$ is now supported in a fixed compact set $K$ for all $\omega$.  Set $F_{m}(\omega,.)= \one_{A_{m}}(\omega)F(\omega,.)$, with $A_{m}=\{\omega ; \|F(\omega,.)\|_{\T^p}\le m\}$. Clearly $F_{m}$ converges to $F$ in  $\L^{p}(\Omega; \T^{p})$ as $m\to \infty$. Moreover, 
 by Lemma 3.1 in \cite{Amenta}, $\|F(\omega,.)\|_{\T^p} \eqsim \|F(\omega,.)\|_{\L^2(K)}\eqsim \|F(\omega,.)\|_{\T^2}$ with implicit constants that depend on $K$. It follows that $$\|F_{m}\|_{\L^\infty(\Omega;\T^2)} \lesssim \|F_{m}\|_{\L^\infty(\Omega;\T^p)} \le m
 $$
and this implies the desired density.\footnote{Identifying   the subspace of  $\T^p$ consisting of functions supported in a compact set $K$ with $\L^2(K)$, we have shown that $\L^\infty(\Omega; \L^2(K))$ is a subspace of $\L^p(\Omega;\T^p)$ for all $1<p<\infty$. And if $\T^p_{c}$ is the subspace of $\T^p$ consisting of functions having compact support, then it is indentified with $\L^2_{c}(\R_{+}\times \R^n)$ and we have shown that $\L^\infty(\Omega; \L^2_{c}(\R_{+}\times \R^n))$ is dense in $\L^{p}(\Omega; \T^{p})$, with a universal procedure working  simultaneously if $F$ happens to belong to several $\L^p(\Omega;\T^p)$.}   
\end{proof}

The last preliminary to our main result is the following proposition, which gives the required  maximal regularity estimate for the stochastic heat equation.

\begin{prop}
\label{lem:ito}
Let $1<p<\infty$. Then, for all simple  adapted processes $g$, 
\begin{align}\label{eq:Tpstoch}
\mathbb{E}&\bigg\|(t,x)\mapsto \bigg( \int  _{0} ^{t} \nabla e^{(t-s) \Delta} g(s,.)\, \d W_{\H}(s)\bigg) (x)\bigg\|_{\T^{p}(\mathbb{C}^{n})} ^{p}  
\lesssim \|g\|^{p} _{\L^p(\Omega;  \T^{p}(\H))}.
\end{align}
 Consequently, the inequality extends to all adapted $g \in \L^p(\Omega;  \T^{p}(\H))$ by density.
\end{prop}
\begin{proof}
Let $g$ be a simple adapted process as in \eqref{eq:sap}.
For $s>0$, define 
$$\tilde{h}(s)\colon (t,x)\mapsto 1_{[0,t]}(s) \nabla e^{(t-s) \Delta} g(s,.)(x).$$
 Reasoning as in the beginning of the proof of Lemma \ref{lem:Smap}, and using the square function estimates for $\nabla e^{(t-s)\Delta}$ applied to  $g(s,.) \in \L^{2}(\Omega; \L^{2}(\R^{n};\H))$ for each $s$ (see the beginning of Section \ref{sec:quadratic}),  one can see that   $s \mapsto \tilde h(s)$ is an adapted 
$\L^{2}(\R_{+}\times \R^{n};\IC^n \otimes \H)$-valued process and that
 $$
 \bigg(\int  _{0} ^{t} \nabla e^{(t-s) \Delta} g(s,.)\, \d W_{\H}(s)\bigg) (x)=\bigg(\int_{0}^{\infty}   \tilde{h}  (s)\,  \d W_{\H}(s)\bigg) (t, x).$$
 Proceeding as in the proof of Lemma \ref{lem:Smap} for $\tilde h $ in 
 the  UMD space $\T^p( \IC^n)$, and interpreting back in terms of $g$, 
gives us
\begin{align*}
\mathbb{E}&\bigg\|(t,x)\mapsto \bigg( \int  _{0} ^{t} \nabla e^{(t-s) \Delta} g(s,.)\, \d W_{\H}(s)\bigg) (x)\bigg\|_{\T^{p}(\mathbb{C}^{n})} ^{p} 
\\ & \quad
 \eqsim \mathbb{E} \bigg\|(t,x) \mapsto \bigg(\int _{0} ^{t} |\nabla e^{(t-s) \Delta}g(s,.)(x)|_{\IC^n\otimes \H}^{2}\, \d s\bigg)^{\frac{1}{2}}\bigg\|_{\T^{p}} ^{p}.
\end{align*}
Given that $e^{t \Delta}$ acts on $\H$-valued functions component-wise, we can apply Theorem \ref{thm:quad1} and Theorem \ref{thm:quad2} to obtain that
$$
\mathbb{E} \bigg\|(t,x) \mapsto \bigg(\int _{0} ^{t} |\nabla e^{(t-s) \Delta}g(s,.)(x)|_{\IC^n\otimes \H}^{2}\, \d s\bigg)^{\frac{1}{2}}\bigg\|_{\T^{p}} ^{p} \lesssim \|g\|^{p} _{\L^p(\Omega;  \T^{p}(\H))}.
$$
This proves \eqref{eq:Tpstoch} for such $g$, and  we  conclude by density. 
\end{proof}

\begin{rem}
 By identification of $\T^2$ with $\V^2$, the estimate for $p=2$   is simply Da Prato's result. 
\end{rem}

\subsection{Main result}
We can now state the main theorem of this section, using at some point the following space of test functions
$$
\Phi^{p} := \{\varphi \in \mathcal{D}'(\R_{+}\times \R^{n}) \;;\; \varphi \in \E^{p},\; \partial_{t}\varphi \in \E^{p}, \; |\nabla \varphi| \in \E^{p}\},$$
with 
$$\|\varphi\|_{\Phi^{p}}= \|\varphi\|_{\E^{p}}+ \|\partial_{t}\varphi\|_{\E^{p}}+ \| |\nabla \varphi|\|_{\E^{p}}.
$$

Note that $\Phi^{p} \hookrightarrow C^{1/2}([0,\infty);\L^{p})$ as proven in Proposition \ref{prop:testclass}, so  in particular it follows that  for $\varphi\in \Phi^p$, $\varphi(0,.)$ exists and  $\|\varphi(0,.)\|_{\L^{p}} \lesssim \|\varphi\|_{\Phi^p}$. Note also that $\L^p(\Omega;\Phi^p)\cap \L^2(\Omega;\Phi^2)$ is dense in $\L^p(\Omega;\Phi^p)$ by Proposition \ref{prop:densetest}.

\begin{thm}
\label{thm:spde}
Let $1<p<\infty$ and let  $(\psi, F,g) \in  \L^{p}(\Omega; \mathcal{T}^{p})$ with $g$  
adapted.  There exists a unique $U \in \L^{p}(\Omega;\dot{\mathcal{T}}^{p}_{1})$  that is a pathwise weak solution of Problem \eqref{eq:SPDE} in the sense of Definition \ref{def:pathweak}. 
Moreover, for each $T \geq 0$, we have continuity with respect to the data:
\begin{equation}
\label{eq:UandgradUest}
\|(t,x)\mapsto \frac{\one_{[0,T]}(t)}{T^{1/2}}U(t,x)\|_{\L^{p}(\Omega;\T^{p})} +
\|\nabla U\|_{\L^{p}(\Omega;\T^{p}(\IC^n)
)} \lesssim \|(\psi,F,g)\|_{\L^{p}(\Omega; \mathcal{T}^{p}}
\end{equation}
with implicit constant independent of $T$.  
 
For 
this solution,  
the property $(\PW)$  of Definition \ref{def:pathweak} extends  to  all test functions in  $  \L^{p'}(\Omega;\Phi^{p'})$.  Moreover,  for any such $\varphi$ and almost all $\omega \in \Omega$, 
$t\mapsto V_{\varphi} (\omega,t) $ has a (unique)  representative $\tilde V_{\varphi} (\omega,t) $ that is absolutely continuous on $[0,\infty)$,  with
\begin{align*}
\tilde V_{\varphi} (\omega,t)  &= \mathbb{V}_{\varphi}(\omega,t) \quad \forall t > 0,
\\
\underset{t \to 0}{\lim} \, \tilde V_{\varphi} (\omega,t) &=  \langle \psi(.,\omega), \varphi(0,.,\omega) \rangle.
\end{align*}

\end{thm}

Recall that talking about pathwise weak solutions in this context is meaningful, as shown in Lemma \ref{lem:pwsTp}.  The extension of $(\PW)$  then  involves the largest space of (random) test functions for which the weak formulation makes sense by duality. Contrary to the definition used in \cite{pronk-veraar, dhn}, we do not need to impose that $\div a(\omega, t,.)^* \nabla \varphi(\omega, t,.)$   and $\partial_{t} \varphi(\omega, t,.)$ belong to $\L^{p'}$ for all $t,\omega$, which would be unrealistic as we do not know the $\L^{p'}$ domains of  $\div a(\omega, t,.)^* \nabla$.

\begin{proof} 
We prove uniqueness first. We then construct a candidate solution with required estimates via a specific representation. Lastly, we prove it is a pathwise weak solution, and we verify $(\PW)$  with the extended class of test functions,  as well as  the  absolutely continuous redefinition,  in two steps. 
First, for $p=2$,  we exploit the $\L^2$ theory of the terms appearing in our representation. 
Secondly, when $p\ne  2$,  we use an approximation argument by density of $\L^{p}(\Omega; \mathcal{T}^{p}) \cap \L^{2}(\Omega; \mathcal{T}^{2})$ in $\L^{p}(\Omega; \mathcal{T}^{p})$.\\

\paragraph {\bf Step 1: Uniqueness.} Let $U,V \in \L^{p}(\Omega;\dot{\mathcal{T}}^{p}_{1})$ be two  pathwise weak solutions of Problem \eqref{eq:SPDE} in the sense of Definition \ref{def:pathweak},  and set  $\widetilde{U} = U-V \in \L^p(\Omega;\dot{\mathcal{T}}^{p}_{1})$. We shall show that, almost surely, $ \widetilde{U}$ is a global weak solution of $
\partial_{t} \widetilde{U} = \div\, a\nabla \widetilde{U}$, and conclude that $\widetilde U$ vanishes by applying a uniqueness result for such  deterministic problems. 

Let  $\varphi \in \C^{\infty}_{c}([0,\infty)\times \R^{n})$. 
   We follow $\omega$ in the notation from now on for this step. 
Using $(\PW)$ and by difference, there is a set $E_{\varphi}$ of probability $1$ such that, for any $\omega \in E_{\varphi}$,  
one has for  almost all $t>0$,
\begin{align}\label{eq:uniqueness2}
&\langle \widetilde{U} (\omega,.,t) ,\varphi(t,.) \rangle \\
& \nonumber \qquad =  - \int  _{0} ^{t} \langle
a (\omega,.,s) \nabla \widetilde{U} (\omega,.,s),\nabla \varphi(s,.) \rangle\, \d s +  \int  _{0} ^{t} \langle
\widetilde{U} (\omega,.,s) ,\partial_{s}\varphi(s,.) \rangle\, \d s. 
\end{align}
In particular, choosing $t$ large,  and recalling that  the $t$-support of $\varphi$ is bounded, we have obtained, for all $ \omega\in E_{\varphi}$,
\begin{equation}\label{eq:uniqueness3}
0 =  - \int  _{0}^{\infty} \langle
a (\omega,.,s) \nabla \widetilde{U} (\omega,.,s) ,\nabla \varphi(s,.) \rangle\, \d s +  \int  _{0} ^{\infty} \langle
\widetilde{U} (\omega,.,s),\partial_{s}\varphi(s,.) \rangle\, \d s. \end{equation}
So far, we  have  assumed  $\varphi \in \C^{\infty}_{c}([0,\infty)\times \R^{n})$.
Now, we let $E=\cap E_{\varphi}$ over $\varphi\in D$ where $D$ is the countable set built in Lemma \ref{lem:countable}, so that $E$ has full measure.  The above equality \eqref{eq:uniqueness3} holds therefore for all $\omega\in E$ and all $\varphi\in D$. 

Let us fix $\omega \in E$   and  we may also assume that  $\widetilde U (\omega,.)  \in  \L^{2}_{\loc} (\R_{+}; \H^1_{\loc}(\R^{n}))$  by Remark \ref{rem:l2loc}. We can deduce  from Lemma \ref{lem:countable}   that  $ \widetilde{U} (\omega,.) $ is a global weak solution of
$
\partial_{t} \widetilde{U} (\omega,.)  = \div\, a (\omega,.) \nabla \widetilde{U} (\omega,.) ,
$
in the sense of \cite[Defn.~1.1]{amp}.

 Let us now deduce that  $\widetilde{U} (\omega,.) =0$. Since $a (\omega,.) $ has real entries, we have that the propagator $\Gamma(\omega,t,s)$ of the equation with coefficients $A=a (\omega,.) $ as in Section  \ref{sec:ing}  
 belongs to  $\B(\L^{q}(\R^{n}))$  for all $q \in [1,\infty]$ and all $t\ge s\ge 0$, by Aronson's result \cite{aronson}, see also \cite[Section 4.3]{amp}. Using  $\nabla  \widetilde{U} (\omega,.)  \in \T^{p}$, we can thus apply Theorem \ref{thm:CP} when $p=2$ and \cite[Thms.~1.1 and~1.5]{zaton} for general $p$  to deduce that there exists $f (\omega,.)  \in \L^{p}$ such that $\widetilde{U} (\omega,.,t)  = \Gamma (\omega,t,0)  f (\omega,.) $ for all $t>0$. In particular, by \cite[Cor. 5.10 and Prop. 5.11]{amp},  $t\to \widetilde{U} (\omega,.,t) $ is continuous as a map from $[0,\infty)$ to $\L^{p}$. Coming back to \eqref{eq:uniqueness2} with an arbitrary $\varphi \in \C^{\infty}_{c}([0,\infty)\times \R^{n})$,  and letting $t\to 0$  along a subsequence, we have that
$$
\langle f (\omega,.),\varphi (0,.)\rangle = 0. $$
As $\varphi (0,.)$ can be an arbitrary element in $\C_{c}^{\infty}(\R^{n})$, 
we conclude that $f (\omega,.) =0$ and thus $\widetilde{U} (\omega,.) =0$.\\

\

\paragraph{\bf Step 2: Construction of a solution with tent space estimates.}

To prove existence, we construct $U$ by the ansatz  $V_{0}+V_{1}+V_{2}$  for the random functions defined for all $t>0$ by 
\begin{align*}
V_{0}(t,.) &=  e^{t\Delta}\psi,   \\
V_{1}(t,.) &=  \int  _{0} ^{t} e^{(t-s)\Delta} g(s, .) \, \d W_{\H}(s), \\
V_{2}(t,.) & = \int  _{0} ^{t} \Gamma(t,s)\div F(s,.)\, \d s +
\int  _{0} ^{t} {\Gamma(t,s)} \div (a(s,.)- I)\nabla (V_{0}+V_{1})(s, .)\, \d s,
\end{align*}
where, again,  $\Gamma(\omega, t,s)$ is for almost all $\omega$ the propagator defined  in Section~\ref{sec:ing} with $A(t,x)=a(\omega, t,x)$. Formally, $U$ is a  solution to \eqref{eq:SPDE}, and $(\PW)$ will be verified in the next steps. In this step, we  prove the tent spaces estimates  for $V_{0}, V_{1}, V_{2}$ and  add them  up.
 Lemma \ref{lem:approxdata} allows us to consider  data $(\psi, F, g)\in \L^2(\Omega;\mathcal{T}^2)\cap \L^p(\Omega;\mathcal{T}^p)$ with $g$ a simple adapted process. With such data, calculations are justified by the $\L^2$ theory. We shall see that, for $p=2$, calculations also hold when $g$ is $\L^2(\Omega;\T^2(\H))$ adapted. Once the estimates are obtained, the conclusion comes by density,  using also Proposition \ref{lem:ito}.   \\

We first use the classical  Littlewood-Paley estimate \eqref{eq:LPest} to get   
\begin{equation}
\label{eq:v0gradest}
\mathbb{E}\|\nabla V_{0}\|_{\T^{p}(\IC^{n})} ^{p} = 
\mathbb{E}\|(t,x)\mapsto \nabla e^{t \Delta}\psi(x)\|_{\T^{p}(\IC^{n})} ^{p} \lesssim  \mathbb{E} \|\psi\|_{\L^{p}} ^{p}.
\end{equation}
Using the $\L^p$ boundedness of the non-tangential maximal function defined by
$$\mathcal{N}(V_{0})(x):= \underset{|y-x|<\sqrt{t}}{\sup} |e^{t \Delta}\psi(y)|,$$ we also have the estimate
\begin{equation}
\label{eq:v0est}
\mathbb{E}\|(t,x)\mapsto 1_{[0,T]}(t)V_{0}(t,x)\|_{\T^{p}} ^{p} 
\lesssim
\mathbb{E}\|T^{1/2} \mathcal{N}(V_{0})\|_{\L^{p}} ^{p} \lesssim T^{p/2}\mathbb{E} \|\psi\|_{\L^{p}} ^{p}.
\end{equation}

Next, for simple adapted processes it is easy to check that 
\footnote{For  $g \in \L^2(\Omega;\T^2(\H))$,  $g$ adapted, this formula holds also in $\L^2(\Omega;\T^2(\H))$ by \cite[Lemma 2.4.1]{rl} and the identification between $\T^2(\H)$ and $\L^2(\R_{+}; \L^2(\R^n; \H))$. Hence this step extends naturally to data $(\psi, F, g)\in \L^2(\Omega;\mathcal{T}^2)$ with $g$ adapted.} 
$$
 \nabla V_{1}(t,.) =  \int  _{0} ^{t} \nabla e^{(t-s)\Delta} g(s, .) \, \d W_{\H}(s) \quad \forall t>0,
 $$
 and the fact that  $V_{1} \in \L^{p}(\Omega;\dot{\mathcal{T}}^{p}_{1})$, that is, 
\begin{equation}
\label{eq:v1gradest}
  \mathbb{E}\|\nabla V_{1}\|^{p}_{\T^{p} (\IC^n)} \lesssim \mathbb{E}
\|g\|_{\T^{p}(\H)}^{p},
\end{equation}
then follows from Proposition \ref{lem:ito}.    The same reasoning  as in Proposition \ref{lem:ito}  but without the gradient,  gives us 
\begin{equation}
\label{eq:v1est}
  \mathbb{E}\|(t,x)\mapsto \one_{[0,T]}(t)V_{1}(t,x)\|^{p}_{\T^{p} } \lesssim T^{p/2}\mathbb{E}
\|g\|_{\T^{p}(\H)}^{p}.
\end{equation}

It remains to look at $V_{2}$. For almost all $\omega \in \Omega$, $V_{2} (\omega,.) $ is the global weak solution in the sense of Theorem~\ref{thm:CP} of the equation $$\partial_{t} V_{2} (\omega,.)  - \div a (\omega,.) ) \nabla V_{2} (\omega,.) =\div F (\omega,.) 
+ \div (a (\omega,.) - I) \nabla (V_{0}+V_{1}) (\omega,.) $$
 with $F (\omega,.) 
+  (a (\omega,.) - I) \nabla (V_{0}+V_{1}) (\omega,.) \in \L^2(\R_{+};\L^2(\R^n,\IC^n))$ and zero initial value (we use here $(\psi,F,g) \in \L^{2}(\Omega;\mathcal{T}^{2})$ and the estimate just obtained with $p=2$). We  may  use Theorem \ref{thm:main1} and Theorem \ref{thm:main2} to obtain 
\begin{equation}
\label{eq:v2est}
\mathbb{E}\| {T^{-1/2}}{\one_{[0,T]}}V_{2}\|_{\T^{p}}^{p}+ \mathbb{E}\|\nabla V_{2}\|_{\T^{p}(\mathbb{C}^{n})}^{p} \lesssim \mathbb{E}
\|F+ 
(a- I)\nabla (V_{0}+V_{1})\|_{\T^{p}(\mathbb{C}^{n})}^{p}. 
\end{equation}
Using \eqref{eq:v0gradest}, \eqref{eq:v1gradest}, \eqref{eq:v2est} and then \eqref{eq:v0est}, \eqref{eq:v1est},   combined with the inclusion \eqref{eq:L1loc},  we have thus shown that $U=V_{0}+V_{1}+V_{2}$  belongs to $\L^{p}(\Omega;\dot{\mathcal{T}}^{p}_{1})$, and satisfies  the required estimates \eqref{eq:UandgradUest}.

 \

  \paragraph{\bf Step 3: Proof that $U$ is a pathwise  weak solution when   $(\psi,F,g) \in 
  \L^{2}(\Omega;\mathcal{T}^{2})$, $g$ adapted.}   Consider  $U=V_{0}+V_{1}+V_{2}$ just constructed (see the footnote in Step 2). In this step, we not only  
 prove that $U$ is a pathwise weak solution of Problem \eqref{eq:SPDE}, but we establish 
 the stronger fact announced in the statement  that  the extended property   $(\PW)$  holds  for such data   
  and $ \varphi \in \L^{2}(\Omega;\Phi^{2}) $,
namely the equality  $ V_{\varphi}=\mathbb{V}_{\varphi}$ in $\L^2(\Omega\times [0,T])$ for all $T>0$ (not just in $\L^1$). We do it for each $V_{i}$ and the claim follows by linearity. \\

 \paragraph{\em Step 3.1: Equality for  $V_{0}$.}   Fix $ \varphi \in \L^{2}(\Omega;\Phi^{2}) $. 
As $V_{0}= e^{t \Delta}\psi$ and $\psi \in \L^2(\Omega; \L^2(\R^n))$, we have $V_{0}\in \L^2(\Omega; \L^2(\R_{+}; \Hdot^1(\R^n))\cap \C_{0}(\R_{+}; \L^2(\R^n)))$ and easily obtain in $\L^2(\Omega)$ that for all    
 $t>0$, 
\begin{align*}
&\langle V_{0}(t,.),\varphi(t,.) \rangle 
- \langle \psi, \varphi(0,.) \rangle
\\ 
&\quad 
= - \int  _{0} ^{t} 
\langle \nabla V_{0}(s,.), \nabla \varphi(s,.)\rangle\,  \d s + \int  _{0} ^{t} 
\langle V_{0}(s,.), \partial _{s} \varphi(s,.) \rangle\,  \d s.\end{align*}
In particular, the equality also holds in $\L^2(\Omega\times [0,T])$ for all $T>0$. \\

\paragraph{\em Step 3.2: Equality for  $V_{1}$.  } It is organised in three substeps.  \\

\subparagraph{Step 3.2.1: Calculations when $g$ is a simple adapted process.}  For such a $g$, we first remark that $\Delta g \in \L^2(\Omega;\T^2(\H))= \L^2(\Omega; \L^2(\R_{+}; \L^2(\R^n;\H)))$, so that  for all $t>0$ , the following equalities hold in $\L^{2}(\Omega; \L^2(\R^n))$,
\begin{align*}
V_{1}&(t,.) - \int _{0} ^{t} g(s,.)\, \d W_{\H}(s) =  \int _{0} ^{t} (e^{(t-s) \Delta}-I)\, g(s,.)\, \d W_{\H}(s)\\
& =  \int _{0} ^{t} \bigg(\int _{s} ^{t} \Delta e^{(\tau-s) \Delta}g(s,.)\, \d \tau\bigg) \, \d W_{\H}(s) =  \Delta \int _{0} ^{t} \bigg(\int _{s} ^{t} e^{(\tau-s) \Delta}g(s,.)\, \d \tau\bigg) \, \d W_{\H}(s),
\end{align*}
where we commute $\Delta$ with the Bochner integral in $\tau$ using a classical theorem of Hille (see \textit{e.g.,} \cite[Thm 1.33]{hnvw1}), and with the It\^o integral as in Step 2 with the gradient (see also \cite[Lemma 2.4.1]{rl}).
 Therefore, for all $\phi \in \C^{\infty}_{c}(\R^{n})$, the stochastic Fubini theorem (see \cite{dapZ}), gives us that  
\begin{align*}
\bigg\langle V_{1}(t,.) &- \int _{0} ^{t} g(s,.)\, \d W_{\H}(s), \phi \bigg\rangle 
= \bigg\langle \int _{0} ^{t} \bigg(\int _{s} ^{t} e^{(\tau-s) \Delta}g(s,.)\, \d \tau\bigg) \, \d W_{\H}(s), 
\Delta \phi \bigg\rangle, \\
& = \bigg\langle \int _{0} ^{t} \bigg(\int _{0} ^{\tau} e^{(\tau-s) \Delta}g(s,.)\, \d W_{\H}(s)\bigg) \, \d \tau, 
\Delta \phi \bigg\rangle= \int _{0} ^{t} \langle V_{1}(\tau,.), 
\Delta \phi \rangle \, \d \tau,
\end{align*}
using Fubini theorem again in the last equality. 
 Since  $V_{1} \in \L^{2}( \Omega \times [0,T];   \H^{1}(\R^{n}))$ for all $T>0$   by \eqref{eq:v1gradest} and \eqref{eq:v1est}  when $p=2$, and since $\phi \in \C^{\infty}_{c}(\R^{n})$, we can integrate by parts in the right hand side, and obtain in $\L^2(\Omega)$ for all $t>0$:
\begin{align}
\label{eq:v1rep}
\langle V_{1}(t,.) - \int _{0} ^{t} g(s,.)\, &\d W_{\H}(s), \phi \rangle =  
- \int _{0} ^{t} \langle  \nabla V_{1}(\tau,.), 
\nabla \phi \rangle \, \d \tau.
\end{align}
\subparagraph{Step 3.2.2: Extension of \eqref{eq:v1rep}  to  adapted  $g \in \L^2(\Omega;\T^2(\H))$.} We proceed by density. First, it follows from the  It\^o isometry that 
$$
\mathbb{E}\bigg|\bigg\langle \int _{0} ^{t} g(s,.)\, \d W_{\H}(s), \phi \bigg\rangle\bigg|^{2} \lesssim 
\|\phi\|^{2} _{\L^{2}}\,  \mathbb{E} \int _{0} ^{t} \|g(s,.)\|_{\L^{2}(\R^{n};\H)}^{2} \, \d s
\leq \|\phi\|^{2} _{\L^{2}} \|g\|^{2} _{\L^2(\Omega;\T^2(\H))},
$$
for all  $t>0$. Next, as pointed out  in the footnote in Step 2,  $V_{1}(t,.)$ and $\nabla V_{1}(t,.)$ make sense for all adapted processes in $g \in \L^2(\Omega;\T^2(\H))=\L^2(\Omega;\V^2(\H))$ as elements of $\L^2(\Omega; \L^2(0,T; \L^2(\R^n))$ for all $T>0$ and  $\L^2(\Omega; \L^2(\R_{+}; \L^2(\R^n;\IC^n))$. Thus, the equality follows in $\L^2(\Omega)$ for almost every $t>0$. Moreover, for all $t>0$,  we have that
$$
\mathbb{E} \bigg|\int _{0} ^{t}  |\langle  \nabla V_{1}(\tau,.), 
\nabla \phi \rangle| \, \d \tau \bigg|^{2} \lesssim t \|\nabla \phi\|^{2} _{\L^{2}(\R^{n};\mathbb{C}^{n})}
 \|g\|^{2} _{\L^2(\Omega;\T^2(\H))},
$$
by \eqref{eq:v1gradest} for $p=2$.  

In particular, this shows that,  for all adapted processes $g \in \L^2(\Omega;T^2(\H))$, the left hand side of equality \eqref{eq:v1rep}  is  (redefined as an) absolutely continuous $\L^2(\Omega)$-valued function with derivative $- \langle  \nabla V_{1}(t,.), 
\nabla \phi \rangle$ for almost every $t>0$.
\Bk\\

\subparagraph{Step 3.2.3: Obtention of the desired formula for   all test functions in $\L^{2}(\Omega; \Phi^{2})$.} We  argue as in \cite[Theorem 4.9]{pronk-veraar} for any fixed  adapted process $g \in \L^2(\Omega;\T^2(\H))$.  
 Consider test functions of the form $\varphi \colon  (\omega,t,x)  \mapsto f(t)\phi(x)$ for $f \in \C^{1}([0,\infty))$ and $\phi \in  \C^\infty_{0}(\R^{n})$.    As   
$
t \mapsto  \langle V_{1}(t,.) - \int _{0} ^{t} g(s,.)\, \d W_{\H}(s), f(t) \phi  \rangle 
$
is  absolutely continuous  as an  $\L^{2}(\Omega)$-valued function, a calculation yields
\begin{align*}
&\qquad \langle  V_{1}(t,.) - \int _{0} ^{t} g(s,.)\, \d W_{\H}(s), \varphi(t,.) \rangle = \\ \nonumber & 
\int _{0} ^{t} \left(  - \langle
 \nabla V_{1}(s,.),\nabla \varphi(s,.) \rangle + \bigg\langle V_{1}(s,.) - \int _{0} ^{s} g(\tau,.)\, \d W_{\H}(\tau), \partial_{s}\varphi(s,.) \bigg\rangle \right)\,  \d s.
\end{align*}
Reorganising terms gives  for all $t>0$, the $ \L^2(\Omega)$ equality
\begin{align}
\label{eq:v1fullrep}
\langle & V_{1}(t,.),\varphi(t,.) \rangle - \bigg\langle \int  _{0} ^{t} g(s,.)\, \d W_{\H}(s), \varphi(t,.) \bigg\rangle=  
- \int  _{0} ^{t} \langle
 \nabla V_{1}(s,.),\nabla \varphi(s,.) \rangle \,  \d s \\ &
 \nonumber \qquad + \int  _{0} ^{t} \langle
V_{1}(s,.),\partial_{s}\varphi(s,.) \rangle\,  \d s - \int  _{0} ^{t} \bigg\langle \int  _{0} ^{s} g(\tau,.)\, \d W_{\H}(\tau), \partial_{s}\varphi(s,.) \bigg\rangle \,\d s.
\end{align}
Now for $T>0$,  remark that, using It\^o isometry and \eqref{eq:v1gradest} and \eqref{eq:v1est} again, each term of this equality belongs to $\L^2(\Omega \times [0,T])$ when  $\varphi\in \L^2(\Omega;  \H^1([0,T]\times \R^n))$. But,  using  density first, the above equality extends to all $\varphi$ being a constant function of $\omega$ valued in $\H^1([0,T]\times \R^n)$. Then by linearity it extends to all  simple functions on $\Omega$ valued in $\H^1([0,T]\times \R^n)$ and by density once again to $\L^2(\Omega;  \H^1([0,T]\times \R^n))$ which clearly contains the restrictions to  $\Omega\times [0,T]\times \R^n$ of
the functions in $\L^{2}(\Omega; \Phi^{2})$.   \\

\paragraph{\em Step 3.3: Equality for  $V_{2}$.} We already have seen that, almost surely, $V_{2}$ is the global weak solution in the sense of Theorem \ref{thm:CP} of 
$$\partial_{t} V_{2} - \div a \nabla V_{2} 
=\div F
+ \div (a- I) \nabla (V_{0}+V_{1}),$$
with $F
+  (a- I) \nabla (V_{0}+V_{1})\in \L^2(\Omega;\L^2(\R_{+}; \L^2(\R^n;\IC^n)))$ and zero initial value. 
In particular, $V_{2}\in \L^2(\Omega; \L^2(\R_{+}; \Hdot^1(\R^n))\cap \C_{0}(\R_{+}; \L^2(\R^n)))$.  
Fix  
$ \varphi \in \L^{2}(\Omega;\Phi^{2})$. 
Recalling \cite[Remark 3.3]{amp}, we thus obtain $t \mapsto \langle V_{2}(t,.),\varphi(t,.) \rangle$ is absolutely continuous on $[0,\infty)$ as an $\L^2(\Omega)$-valued function  and  that 
 for all $t> 0$, 
 \begin{align*}
&  \langle V_{2}(t,.),\varphi(t,.) \rangle=  \int  _{0} ^{t} 
\langle \partial _{s} V_{2}(s,.), \varphi(s,.) \rangle\, \d s + \int  _{0} ^{t} 
\langle V_{2}(s,.), \partial _{s} \varphi(s,.) \rangle\, \d s\\
& \quad = - \int  _{0} ^{t} \langle a(s,.)\nabla V_{2}(s,.), \nabla \varphi(s,.) \rangle \,\d s  
- \int  _{0} ^{t} 
\langle  F(s,.), \nabla \varphi(s,.) \rangle\, \d s \\
&- \int  _{0} ^{t} 
\langle (a(s,.)- I)\nabla (V_{0}+V_{1})(s,.), \nabla \varphi(s,.) \rangle\, \d s
+ \int  _{0} ^{t} 
\langle V_{2}(s,.), \partial _{s} \varphi(s,.) \rangle\, \d s. 
\end{align*}
Again, one can see the equality in $\L^2(\Omega\times [0,T])$ for all $T>0$.

\

\paragraph {\bf Step 4: Proof of $(\PW)$ and regularity when $(\psi,F,g) \in \L^{p}(\Omega; \mathcal{T}^{p})$, $g$ adapted.} 

We now consider a general data $(\psi,F,g) \in \L^{p}(\Omega; \mathcal{T}^{p})$, $g$ adapted. 
First, by Lemma \ref{lem:approxdata}, we may approximate $(\psi,F,g) $ in $ \L^{p}(\Omega; \mathcal{T}^{p})$  by  $(\psi_{\ell},F_{\ell},g_{\ell}) \in \L^{p}(\Omega; \mathcal{T}^{p})\cap \L^{2}(\Omega;\mathcal{T}^{2})$, $\ell\ge1$, with $g_{\ell}$ simple adapted processes. With this approximation, we have the estimates of Step 2 for the constructed functions $U_{\ell}$  with data $(\psi_{\ell},F_{\ell},g_{\ell})$    and the same ones passing to the limit in \eqref{eq:UandgradUest} for the limit $U$.  We claim  that the  formula in $(\PW)$ for $U$  holds in $\L^1(\Omega\times [0,T])$ for all $T>0$, which   implies  that   $U$ is a pathwise weak solution of Problem  \eqref{eq:SPDE}. To this end, we take $\varphi \in \L^{p'}(\Omega;\Phi^{p'})$ and set  $W_{ \varphi}(U,\psi,F,g)\coloneqq V_{\varphi}-\mathbb{V}_{\varphi}$, so that  we  have to show for any $T>0$ that $W_{ \varphi}(U,\psi,F,g)=0$ in $\L^1(\Omega\times [0,T])$.

We also approximate $\varphi$ in $\L^{p'}(\Omega;\Phi^{p'}) $ norm by a sequence $\varphi_{m} \in
\L^{p'}(\Omega;\Phi^{p'}) \cap \L^{2}(\Omega;\Phi^{2})$  
 as in Proposition \ref{prop:densetest}. 
  We fix $T>0$ and    prove that  
\begin{align*}
\alpha_{\ell,m}&\coloneqq\mathbb{E} \int _{0} ^{T} |W_{\varphi_{m}}(U, \psi,F,g)(t)-W_{\varphi_{m}}(U_{\ell},\psi_{\ell},F_{\ell},g_{\ell})(t)|\,  \d t 
\underset{\ell \to \infty}{\to} 0, \ \   \forall m\in \IN, \\
\beta_{m}&\coloneqq\mathbb{E} \int _{0} ^{T} |W_{ \varphi}(U,\psi,F,g)(t)-W_{\varphi_{m}}(U,\psi,F,g)(t)|\,  \d t 
\underset{m \to \infty}{\to} 0.
\end{align*}
As  we have proven in Step 3 that 
$W_{\varphi_{m}}(U_{\ell}, \psi_{\ell},F_{\ell},g_{\ell})=0$
 in $\L^2( \Omega  \times [0, T]) $,
for all $\ell,m \in \mathbb{N}$, we conclude  $W_{\varphi}(U, \psi,F,g)=0$ in $\L^1(\Omega\times [0,T])$, hence the claim. 

We begin with $\alpha_{\ell,m}$.  We use the triangle inequality, and consider the resulting two terms in $V_{\varphi}$ and  five terms in $\mathbb{V}_{\varphi}$  separately. 
Using duality and 
 \eqref{eq:UandgradUest}, we have that
\begin{align*}
&\mathbb{E} \int _{0} ^{T} |\langle U(t,.)-U_{\ell}(t,.),\varphi(t,.) \rangle| \, \d t \\
& \qquad \lesssim 
\|(t,x)\mapsto {1_{[0,T]}(t)}(U(t,x)-U_{\ell}(t,x))\|_{\L^{p}(\Omega;\,\T^{p})}
\|\varphi\|_{\L^{p'}(\Omega;\,\T^{p'})} \\ & \qquad \lesssim T^{1/2}\|(\psi-\psi_{\ell},F-F_{\ell},g-g_{\ell})\|_{\L^{p}(\Omega; \mathcal{T}^{p})}\|\varphi\|_{\L^{p'}(\Omega;\, \T^{p'})}.
\end{align*}
In the same way
\begin{align*}
&\mathbb{E} \int _{0} ^{T} \int  _{0} ^{t} | \langle a(s,.)(\nabla U(s,.)-\nabla U_{\ell}(s,.)),\nabla \varphi(s,.) \rangle| \, \d s\, \d t\\ & \qquad
\lesssim T\|(\psi-\psi_{\ell},F-F_{\ell},g-g_{\ell})\|_{\L^{p}(\Omega; \mathcal{T}^{p})} \||\nabla \varphi|\|_{\L^{p'}(\Omega;\:\T^{p'})},
\end{align*}
and
\begin{align*}
&\mathbb{E} \int _{0} ^{T} \int  _{0} ^{t}  |\langle
(U(s,.)-U_{\ell}(s,.),\partial_{s}\varphi(s,.) \rangle| \,  \d s\,  \d t\\ &
\qquad \lesssim T\|(\psi-\psi_{\ell},F-F_{\ell},g-g_{\ell})\|_{\L^{p}(\Omega; \mathcal{T}^{p})} \|\partial_{s} \varphi\|_{\L^{p'}(\Omega;\:\T^{p'})}.
\end{align*}

 The term involving $\psi$ is straightforward to treat and yields
\begin{align*}
\mathbb{E} | \langle \psi-\psi_{\ell}, \varphi(0,.) \rangle | &\leq \|\psi-\psi_{\ell}\|_{\L^{p}(\Omega;\, \L^{p})}
\|\varphi(0,.)\|_{\L^{p'}(\Omega; \:\L^{p'})} \\
&
 \lesssim \|\psi-\psi_{\ell}\|_{\L^{p}(\Omega;\, \L^{p})}  \|\varphi\|_{\L^{p'}(\Omega;\:\T^{p'})},
 \end{align*}
 using Proposition \ref{prop:testclass} in the last inequality. 
 
  For the term involving $F$, we directly obtain
\begin{align*}
\mathbb{E} \int  _{0} ^{T} 
\int  _{0} ^{t} 
&|\langle F_{\ell}(s,.)-F(s,.), \nabla \varphi(s,.) \rangle| \, \d s\, \d t \lesssim T \|F_{\ell}-F\|_{\L^p(\Omega;\, \T^{p}(\IC^n))}\||\nabla \varphi|\|_{\L^{p'}(\Omega;\,\T^{p'})}. 
 \end{align*}

 For the terms involving $g$, we use  the operator $S$ from the proof of  Lemma \ref{lem:Smap} and duality.  This is where we need our spaces  $\F^p$ and $\E^{p'}$. We have that
\begin{align*}
\mathbb{E} &  \int  _{0} ^{T}\bigg| \bigg\langle \int  _{0} ^{t} (g_{\ell}(s,.)-g(s,.))\, \d W_{\H}(s), \varphi(t,.) \bigg\rangle\bigg| \, \d t\\
& \qquad \lesssim \|S(|g_{\ell}-g|_{\H})\|_{\L^{p}(\Omega;\, \F^{p} )} \|\varphi\|_{\L^{p'}(\Omega; \,\E^{p'} )}\\ 
& \qquad \lesssim T^{1/2} \||g_{\ell}-g|_{\H}\|_{\L^{p}(\Omega; \,\F^{p} )} \|\varphi \|_{\L^{p'}(\Omega; \, \E^{p'} )}. 
\end{align*}
A similar argument gives that
\begin{align*}
\mathbb{E}  \int  _{0} ^{T}\int  _{0} ^{t}\bigg| \bigg\langle \int  _{0} ^{s} (g_{\ell}(\tau,.)-g(\tau,.))\, & \d W_{\H}(\tau), \partial_{s}\varphi(s,.) \bigg\rangle \bigg|\,  \d s \d t \\
&
 \lesssim \T^{3/2} \||g_{\ell}-g|_{\H}\|_{\L^{p}(\Omega; \,\F^{p} )}
  \| \partial_{s}  \varphi \|_{\L^{p'}(\Omega; \, \E^{p'} )}.
\end{align*}
Altogether, using the embeddings $\T^p\hookrightarrow \F^p$ and  $\E^{p'} \hookrightarrow \T^{p'}$, we have thus proven that
\begin{align*}
 \alpha_{\ell,m}
\lesssim  \max (1,T^{3/2})   \|(\psi-\psi_{\ell},F-F_{\ell},g-g_{\ell})\|_{\L^{p}(\Omega; \mathcal{T}^{p})}\|\varphi_{m} \|_{\L^{p'}(\Omega; \,\Phi^{p'})},
\end{align*}
for all $\ell,m \in \mathbb{N}$. 
The above reasoning also gives us
\begin{align*}
 \beta_{m}
\lesssim   \max (1,T^{3/2})  \|(\psi,F,g)\|_{\L^{p}(\Omega; \mathcal{T}^{p})}\|\varphi-\varphi_{m}\|_{\L^{p'}(\Omega; \,\Phi^{p'})}.
\end{align*}
The desired limits follow.

We finish the proof with the time regularity issue. Using similar estimates as in Lemma \ref{lem:pathweak},   
we see that $\mathbb{V}_{\varphi}  \in   \W^{1,1}([0,T]; \L^1(\Omega))$ for any $T>0$. In particular,  $\mathbb{V}_{\varphi}$  is absolutely continuous on $[0,\infty)$ as an $\L^1(\Omega)$-valued function with limit $\langle \psi, \varphi(0,.) \rangle$ at $t=0$. This implies  that $V_{\varphi}(\omega, t,.)$ almost surely has an  absolutely continuous representative on $\R_{+}$   equal to $\mathbb{V}_{\varphi}(\omega, t,.)$,  with the desired limit at $t=0$. 
 \end{proof}

\begin{rem}
When $p<2$, we prove in Step 2 stronger tent space $\T^p$ estimates  than we really need in Step 4 as  $\F^p\,  ( =\V^p)$ estimates on $U$, $F$ and $g$ seem to  suffice.   This is because we do not know  bounds for the $V_{2}$ term if $F\in  \L^p(\Omega; \V^p(\IC^n))$, $p\ne 2$. \end{rem}

\appendix

\section{The $\L^{p}(\Omega;\Phi^{p})$ test function class}

\begin{prop}
\label{prop:testclass}
For $1<p<\infty$, we have that $\Phi^{p} \hookrightarrow \C^{1/2}([0,\infty);\L^p)$ and in particular elements of $\Phi^{p}$ have a  trace in $\L^p$ at $t=0$. 
\end{prop}
\begin{proof}
We first remark that, for $t>\tau>t/4> 0$, and $f \in W^{1,1}([t/4,t])$, we have that
$ f(t) - f(\tau) = \int _{\tau} ^{t} f'(s)\, \d s$, and thus, integrating  $\tau$ in $[t/4,t/2]$, we get
$$
|f(t)| \leq \fint _{t/4} ^{t/2} |f(\tau)| \,\d \tau + 6 \fint _{t/4} ^{t} s|f'(s)| \,\d s.
$$
 Let $\phi \in \Phi^{p}$. We  use this to show first  $\phi(t,.)\in \L^p$ for all $t>0$ and then establish the  $\C^{1/2}$ inequality on $(0,\infty)$, which classically  provides the existence of a limit  as $t\to0$, hence proving the embedding and the trace. 

Case 1: let $2 \leq p<\infty$  and $t>\tau > 0$. First, plugging $f(t)=\phi(t,x)$ in the above inequality, it follows  by Cauchy-Schwarz inequality that 
\begin{align*}
\|\phi(t,.)\|_{\L^{p}} &\lesssim \bigg\|\fint _{t/4} ^{t/2} |\phi(s,.)| \, \d s\bigg\|_{\L^{p}} + \bigg\| \fint _{t/4} ^{t} s|\partial _{s}\phi(s,.)| \, \d s \bigg\|_{\L^{p}} \\
& \lesssim t^{-1/2} \|\phi\|_{\V^{p}} + t^{1/2} \|\partial_{s}\phi\|_{\V^{p}} \lesssim (t^{-1/2}+t^{1/2})\|\phi\|_{\Phi^{p}}.
\end{align*}
Secondly, we also have that
\begin{align*}
\|\phi(t,.)-\phi(\tau,.)\|_{\L^{p}} &\lesssim  \bigg\| \int _{\tau} ^{t} |\partial _{s}\phi(s,.)| \, \d s \bigg\|_{\L^{p}} \\
& \lesssim (t-\tau)^{1/2} \|\partial_{s} \phi\|_{\V^{p}} \leq (t-\tau)^{1/2} \|\phi\|_{\Phi^{p}}.
\end{align*}

Case 2: let $1 < p \leq 2$ and $t>\tau> 0$. Taking $f(t)=\phi(t,x)$ as above, and using Jensen's inequality,  then the averaging technique for integrals on $\R^n$, and again Jensen's inequality as $p\le 2$, we have that
\begin{align*}
&\|\phi(t,.)\|^{p} _{\L^{p}} 
\lesssim \fint _{t/4} ^{t/2} \| \phi(s,.)\|^{p} _{\L^{p}} \, \d s +  \fint _{t/4} ^{t} \| s\partial _{s}\phi(s,.)\|^{p}_{\L^{p}} \, \d s \\
&\quad = \int _{\R^{n}} \fint _{t/4} ^{t/2} \fint_{B(x,\sqrt{s})} |\phi(s,x)|^{p} \, \d y \, \d s \, \d x +    \int _{\R^{n}} \fint _{t/4} ^{t} \fint_{B(x,\sqrt{s})} |s\partial_{s}\phi(s,x)|^{p} \, \d y \, \d s \, \d x \\
& \quad \lesssim t^{-p/2} \|\phi\|^{p} _{\T^{p}} + t^{p/2} \|\partial_{s}\phi\|^{p} _{\T^{p}} \lesssim (t^{-p/2}+t^{p/2})\|\phi\|^{p} _{\Phi^{p}}.
\end{align*}
Similarly, we also have that
\begin{align*}
 \|\phi(t,.)-\phi(\tau,.)\|^{p} _{\L^{p}} &\lesssim (t-\tau)^{p} \bigg\| \fint _{\tau} ^{t} |\partial _{s}\phi(s,.)| \, \d s \bigg\|^{p} _{\L^{p}} \\
& 
\lesssim (t-\tau)^{p} \int _{\R^{n}} \fint _{\tau} ^{t} \fint _{_{B(x,\sqrt{s})}}|\partial_{s}\phi(s,x) |^{p} \, \d y \, \d s \,\d x \\
&\leq (t-\tau)^{p/2} \|\partial_{s}\phi\|^{p} _{\T^{p}} \leq (t-\tau)^{p/2} \|\phi\|^{p} _{\Phi^{p}}.
\end{align*}
\end{proof}

\begin{prop}
\label{prop:densetest}
Let $1 \leq p<\infty$. 
The set $\L^p(\Omega;\Phi^p)\cap \L^2(\Omega;\Phi^2)$ is dense in $\L^p(\Omega;\Phi^p)$.

\end{prop}
\begin{proof} 
By construction of the Bochner integral,  linear combinations of elementary tensors are dense in $\L^p(\Omega;\Phi^p)$. We thus only need to consider one of them 
$(\omega,t,x)\mapsto  \one_{A}(\omega)  \phi(t,x),$
where $A$ is a measurable set in $\Omega$ and $\phi \in \Phi^{p}$.
As those tensors are in $\L^\infty(\Omega; \Phi^p)$, it remains to show that $\Phi^p \cap \Phi^2$ is dense in $\Phi^p$.

The first step is that one can restrict to functions in $\Phi^p$ with compact support in some $[0,T]\times B(0,R)$. Indeed, let $\eta$ be a smooth function on $[0,\infty)$ which is 1 on $[0,1/2]$ and compactly supported in $[0,1]$ and $\chi$ be a smooth compactly supported function on $\R^n$ with $\chi=1$ on the ball $B(0,1/2)$ and support contained in $B(0,1)$.   It is easy to show  that, if $\phi\in \Phi^p$, then $\phi_{T,R}\colon(t,x) \mapsto  \eta(t/T)\chi(x/R)\phi(t,x)$ belongs to  $\Phi^p$, and converges to $\phi$ in $\Phi^p$ when $T,R\to \infty$. 

When $p\ge 2$,  the result follows immediately since, in this case, $\E^p=\V^p$,  and compact support in $x$ implies that $\phi_{T,R}$  and its partial derivatives belong to $\V^2$ by H\"older inequality.

 When $p<2$ (\textit{i.e.,} $\E^p=\T^p$), we regularise the compactly supported $\phi_{T,R}$
by convolution. This  will create bounded functions with compact support.  Such functions belong to all $\T^q$ spaces, which  concludes the proof. 

 Let $\phi\in \Phi^p$ have  compact support in $[0,T]\times B(0,R)$.   Let $(\psi_{\varepsilon})_{0<\varepsilon<1}$ be a usual mollifying sequence in $\R^n$, where $\psi$ is smooth with support in the unit ball and $\int_{\R^n} \psi(x)\, \d x =1$. Set $\phi_{\varepsilon}(t,x)= (\psi_{\varepsilon} \star_{x}\phi(t,.))(x)$ where convolution is in the $x$ variable.  The function $\phi_{\varepsilon}$ is then compactly supported in   $[0,T]\times B(0,R+1)$. Next, 
using Proposition \ref{prop:testclass}, we  have that $\phi \in \C([0,T]; \L^p)$, and thus, for each $\varepsilon>0$,
$$
\sup_{(t,x)\in [0,T]\times B(0,R+1)} |\phi_{\varepsilon}(t,x)| \le \|\psi_{\varepsilon}\|_{\L^{p'}} \sup_{t\in [0,T]}\|\phi(t,.)\|_{\L^p} <\infty.
 $$
 Since  this convolution commutes with taking derivatives in $t$ and $x$, 
 we have that $\phi_{\varepsilon} \in \Phi^{q}$ for all $1<q<\infty$. It just remains to show that  $\|\phi_{\varepsilon}-\phi\|_{\T^p}\to 0$ as $\varepsilon\to 0$,  for $\phi \in \T^p$ with compact support in  $[0,T]\times B(0,R)$. At this point, one can approximate $\phi$ by $\tilde\phi$ 
 as in the proof of  Lemma \ref{lem:approxdata}, truncating $\phi$ so that 
$$\|\tilde\phi_{\varepsilon}-\tilde\phi\|_{\T^p}\eqsim  \|\tilde\phi_{\varepsilon}-\tilde\phi\|_{\L^2(K)}$$
for some  compact set $K\subset \R_{+}\times \R^n$ independent of $\varepsilon$. It remains to let $\varepsilon\to 0$ by well-known regularisation in Lebesgue spaces.  
 \end{proof}

\end{document}